\documentclass[11pt]{article}
\usepackage[numbers,sort&compress]{natbib}
\usepackage{enumerate}
\usepackage{amscd}
\usepackage{amsmath}
\usepackage{latexsym}
\usepackage{amsfonts}
\usepackage{setspace}
\usepackage{amssymb}
\usepackage{amsthm}
\usepackage{verbatim}
\usepackage{mathrsfs}
\usepackage{enumerate}
\usepackage[hypertexnames=false]{hyperref}
\usepackage[utf8]{inputenc}
\usepackage{amsmath, amssymb, pgfplots}
\pgfplotsset{compat = newest, width = 12cm, height =12cm}
\usepackage{tikz}
\usepackage{caption}

\theoremstyle{plain}\newtheorem{definition}{Definition}[section]
\theoremstyle{definition}\newtheorem{theorem}{Theorem}[section]
\theoremstyle{definition}\newtheorem{lemma}[theorem]{Lemma}
\theoremstyle{plain}
\theoremstyle{plain}
\theoremstyle{definition}\newtheorem{remark}{Remark}[section]
\usepackage{xcolor}

\newcommand{\mr}{\mathbb{R}}

\newcommand{\dd}{\mathrm{~d}}

\newcommand{\diver}{\text{div}\,}

\newcommand{\define}{\stackrel{\mathrm{def}}{=}}

\allowdisplaybreaks

\numberwithin{equation}{section}
\begin{document}
	%%%%%%%%%%%%%%%%%%%%%%%%%%%%%%%%%%%%%%%%%%%%%%%%%%%%%%%%%%%%%%%%%%%%%%%%%%%%%%%%%%%%%%%%%%%%%%%%%%%%
	\title{Non-uniform Continuity for the MHD equations with only Magnetic Diffusion}
	\author{Quansen Jiu\footnote{School of Mathematical Sciences, Capital Normal University, Beijing, 100048, P. R. China. Email: jiuqs@cnu.edu.cn}~~~\,\,\,\,Yaowei Xie\footnote{School of Mathematical Sciences, Capital Normal University, Beijing, 100048, P. R. China. Email: mathxyw@163.com}}
	\date{}
	\maketitle
	\begin{abstract}
		In this paper, we prove the non-uniform continuity of the data-to-solution map for the incompressible magnetohydrodynamic (MHD) equations with only magnetic diffusion in Sobolev spaces $H^s(\mathbb{R}^d)$ for all $s>0$ and  $d=2,3$.  Our results are first studies on the  non-uniform continuity of the data-to-solution map for the resistive MHD equations.   Moreover, our results permit the solution perturbation around an arbitrary constant background magnetic fields $\mathbf{B_0} \in \mathbb{R}^d$, which reveal that the strong magnetic  background fields may provide the stabilization effect but still preserve the  analytical feature of non-uniform continuity of the data-to-solution map.
	\end{abstract}
	\noindent {\bf MSC(2020):}\quad 35B30, 35Q35, 76W05.
	\vskip 0.02cm
	\noindent {\bf Keywords:} Non-uniform continuity, the resistive MHD equations, background magnetic fields.

	\tableofcontents
	
	%%%%%%%%%%
	\section{Introduction}
	\,\,\,\,\,\,\,\, In this paper, we consider the Cauchy problem for the incompressible MHD equations with only magnetic diffusion (hereinafter called the resistive MHD equations):
\begin{align}\label{mhd}
	\begin{cases}
	\partial_t u+u \cdot \nabla u+\nabla p=b\cdot \nabla b,\\[1mm]
\partial_t b-\Delta b+u\cdot \nabla b=b\cdot\nabla u,  \\[1mm]
\text{div}\, u=\text{div}\, b=0, \\[1mm]
u(x, 0)=u_0(x),\,\, b(x, 0)=b_0(x),
	\end{cases}
\end{align}
where $u(x,t) \colon \mathbb{R}^d \times \mathbb{R}_+ \to \mathbb{R}^d$,
$b(x,t) \colon \mathbb{R}^d \times \mathbb{R}_+ \to \mathbb{R}^d$,
and $p(x,t) \colon \mathbb{R}^d \times \mathbb{R}_+ \to \mathbb{R}$
denote the velocity field, magnetic field, and pressure field, respectively,
with $x \in \mathbb{R}^d~(d=2,3)$ and $t>0$. Clearly, when $b\equiv0$, \eqref{mhd} reduces to the classical Euler equations.

The resistive MHD equations provide a fundamental framework for modeling key plasma phenomena where finite electrical resistivity plays a critical role, particularly in astrophysical magnetic reconnection processes that govern energy release in solar flares and magnetospheric activity, as well as in the geodynamo mechanisms responsible for generating and sustaining planetary magnetic fields \cite{priest-2000,roberts1967introduction}.

 There have been a large number of mathematical studies on the  well-posedness theory for the incompressible MHD equations under various assumptions on velocity dissipation and magnetic diffusion.
Previous results on local well-posedness can be found in \cite{sermange-1983,jiu-2006-local,li-local-nonresistive-2017-adv,chemin-local-nonresistive-2016-adv,fefferman-local-nonressitive-2014-jfa,fefferman-local-nonressitive-2017-arma}. The pioneering work of Bardos, Sulem and Sulem \cite{full-not-bardos} and Lin and Zhang \cite{lin-2014-GlobalSmallSolutions} on perturbation theory near constant background magnetic fields first revealed the stabilizing role of such equilibrium configurations in the MHD equations. Physical interpretations of this stabilization effect are discussed in \cite{background-phy1,background-phy2}. These foundational results have inspired extensive research on global well-posedness and stability of the MHD equations, as documented in \cite{lei-bkm-2009-dcds,full-1-duvaut-lions,sermange-1983,cao-wu-1-mix,cao-wu-2-mix,full-not-bardos,lin-2014-GlobalSmallSolutions,chen-2022-3dmhd-Diophant,deng-zhang-2018-decay,panGlobalClassicalSolutions2018,ren2014global,zhang-2016-non-resistive-global-jde,xie-2024-cvpde-dio,wei--global-resistive-2020-cmr,zhou-zhu-glbal-symmetry-jmp-2018,ye-global-resistive-2022-acta} and references therein.

  In recent years, ill-posedness theory  for the incompressible MHD equations has also been invetigated. For the non-resistive MHD equations  which contain velocity dissipation but  no magnetic diffusion, Chen, Nie and Ye \cite{chen-ye-sharp-ill-nonresistive-2024-jfa}  established sharp strong ill-posedness results that provide a striking contrast to the local well-posedness theory developed by Fefferman et al. \cite{fefferman-local-nonressitive-2017-arma}. For the resistive case \eqref{mhd}, Wu and Zhao \cite{wu-zhao-mild-resistive-2023-IMRN}  obtained mild ill-posedness results near the background magnetic field $(1,0)$ in $\mr^2$.

%The global well-posedness results established in \cite{full-not-bardos,lin-2014-GlobalSmallSolutions} and related works consistently require the velocity field to be sufficiently small while maintaining the magnetic field close to a non-zero background state. This persistent structural condition across studies strongly indicates that background magnetic fields may play a stabilizing role in the MHD equations, potentially through dispersive effects or nonlinearity suppression mechanisms.

It is noted that the concept of ill-posedness manifests rather strongly in many fundamental cases. As evidenced by \cite{chen-ye-sharp-ill-nonresistive-2024-jfa,wu-zhao-mild-resistive-2023-IMRN} and other works of PDE systems, such behavior typically occurs in critical or supercritical function spaces, while  some problems remain unresolved to this day. These substantial difficulties have motivated researchers to consider relaxed notions of ill-posedness by examining weaker properties.
 For certain equations, the solution operator may exhibit non-uniform continuity properties under stronger topological frameworks, which provides meaningful insights into the refined continuity structure of solution mappings and can be seen a kind of instability of the solutions or ill-posedness the equations.

In this paper, we are concerned with the non-uniform continuity properties of the data-to-solution map for the resistive MHD equations. We begin by precisely defining the notion of non-uniform continuity as follows:
\begin{definition}\label{define-nonuniform}
	Let $X$ be a Banach space, and consider the Cauchy problem:
	\begin{align*}
		\begin{cases}
			\partial_t v = N(v), \\
			v(0) = v_0,
		\end{cases}
	\end{align*}
	where $N$ is a (possibly nonlinear) differential operator. The \textbf{data-to-solution map} $\Phi_t \colon X \to X$ (for fixed $t > 0$) defined by $\Phi_t(v_0) = v(t)$ is said to be \textbf{non-uniformly continuous} on $X$ if the following holds:
	
	\noindent
	\textbf{Sequential Formulation:} \\
	For every $t > 0$, there exists $\epsilon_0 > 0$ and sequences $\{v_{1,n}(0)\}, \{v_{2,n}(0)\} \subset X$ such that:
	\begin{align*}
		\lim_{n \to \infty} \| v_{1,n}(0) - v_{2,n}(0) \|_X = 0,
	\end{align*}
	but
	\begin{align*}
		\limsup_{n \to \infty} 	\| \Phi_t(v_{1,n}(0)) - \Phi_t(v_{2,n}(0)) \|_X \geq \epsilon_0.
	\end{align*}
	
	\noindent
	\textbf{$\delta$-$\epsilon$ Formulation (Equivalent):} \\
	For every $t > 0$, there exists $\epsilon_0 > 0$ such that for any $\delta > 0$, one can find initial data $v_1(0), v_2(0) \in X$ satisfying:
	\begin{align*}
		\| v_1(0) - v_2(0) \|_X < \delta,
	\end{align*}
	but
	\begin{align*}
		\| \Phi_t(v_1(0)) - \Phi_t(v_2(0)) \|_X \geq \epsilon_0.
	\end{align*}
\end{definition}

In \cite{nonuniform-2010-cmp-himonas}, Himonas and Misiołek first  proved non-uniform continuity of the data-to-solution map  on the  incompressible Euler equations in both $H^s(\mathbb{R}^d)$ ($s > 0$) and $H^r(\mathbb{T}^d)$ ($r \in \mathbb{R}$) with $d = 2,3$, which was later extended by Li and Bourgain \cite{nonuniform-2019-cmp-bourgain-li} through Galilean boost techniques to the endpoint case $s \geq 0$ in $H^s(\mathbb{R}^d)$, where they further proved the stronger property of nowhere uniform continuity.
For the non-resistive MHD equations with only velocity dissipation, recent work by Li, Yin, and Zhu \cite{nonuniforem-2023-adv-li-yin-zhu} demonstrated non-uniform continuity in $H^s(\mathbb{R}^d)$ for $s > \frac{d}{2}$.

  We will prove in this paper the non-uniform continuity of the data-to-solution map for the resistive MHD equations in Sobolev spaces $H^s(\mathbb{R}^d)$ for all $s>0$ and  $d=2,3$. In comparison with the work by Li, Yin, and Zhu \cite{nonuniforem-2023-adv-li-yin-zhu} which is for non-resistive MHD equations with only velocity dissipation and in $H^s(\mathbb{R}^d)$ for $s > \frac{d}{2}$,  our results hold for the resistive MHD equations and  in  $H^s(\mathbb{R}^d)$ for all $s>0$ and  $d=2,3$.  Moreover, our results permit the solution perturbation around an arbitrary constant background magnetic fields $\mathbf{B_0} \in \mathbb{R}^d$, which shows that the strong magnetic field may provide the stabilization effect but no help for the uniform continuity of the data-to-solution map.

Our main results can be stated as follows:
\begin{theorem}\label{thm}
	Let $s>0, \lambda>0, d=2,3$ and $T>0$. Then the data-to-solution map $(u_0,b_0)\mapsto(u,b)$ for the equations \eqref{mhd} is non-uniformly continuous from a bounded subset in $H^{s}(\mr^d)\times H^{s}(\mr^d)$ into $C\left([0,T],H^{s}(\mr^d)\right)\times C\left([0,T],H^{s}(\mr^d)\right)$.

  More precisely, for any $\gamma>0$ and arbitrary constant magnetic field ${\bf B_0}\in \mr^d$,  there exists two sequences of solutions $(u^{+1,\lambda},b^{+1,\lambda})$ and $(u^{-1,\lambda},b^{-1,\lambda})$ such that
	\begin{itemize}
		\item the solutions satisfy
\begin{align*}
			(u^{\pm 1,\lambda}, b^{\pm 1,\lambda} - {\bf B_0}) \in C([0,T], H^{s}(\mathbb{R}^{d})) \times C([0,T], H^{s}(\mathbb{R}^{d}));
\end{align*}
		\item  the initial data satisfy
\begin{equation}\label{nonuniform-initial-bound}
	\begin{array}{l}
							\|u^{+1,\lambda}(0,\cdot)\|_{H^{s}} + \|b^{+1,\lambda}(0,\cdot) - {\bf B_0}\|_{H^{s}} \leq \gamma, \\[2mm]
			\|u^{-1,\lambda}(0,\cdot)\|_{H^{s}} + \|b^{-1,\lambda}(0,\cdot) - {\bf B_0}\|_{H^{s}} \leq \gamma;
	\end{array}
\end{equation}
		\item and the non-uniform continuity is characterized by
		\begin{itemize}
			\item at initial time $t=0$,
			\begin{align}\label{nonuniform-initial-diff}
				\lim_{\lambda\to\infty} \|u^{+1,\lambda}(0,\cdot) - u^{-1,\lambda}(0,\cdot)\|_{H^{s}} = 0,
			\end{align}
			\item for evolution times $t>0$,
			\begin{align}\label{nonuniform-low-bound-sint}
				\lim_{\lambda \to\infty} \|u^{+1,\lambda}(t) - u^{-1,\lambda}(t)\|_{H^{s}} \geq c\gamma |\sin t|,
			\end{align}
		\end{itemize}
		where $c = c(s,d) > 0$ is a constant depending only on  $s$ and $d$.
	\end{itemize}
\end{theorem}
\begin{remark}
	The stabilizing role of background magnetic fields in the MHD equations has been well studied in previous works \cite{lin-2014-GlobalSmallSolutions,full-not-bardos}. Our analysis reveals that despite the stabilization effect of the  background magnetic fields, the data-to-solution map for the resistive MHD equations \eqref{mhd} maintains its non-uniform continuity property for arbitrary non-zero constant fields ${\bf B_0}\neq {\bf 0}$.
\end{remark}

To simplify the notations, we write the perturbed system as follows
\begin{align}\label{mhd-B}
	\begin{cases}
		\partial_t u+u \cdot \nabla u+\nabla p=b\cdot \nabla b+{\bf B_0}\cdot\nabla b,\\[1mm]
		\partial_t b-\Delta b+u\cdot \nabla b=b\cdot\nabla u+{\bf B_0}\cdot\nabla u,  \\[1mm]
		\text{div}\, u=\text{div}\, b=0, \\[1mm]
		u(x, 0)=u_0(x),\,\, b(x, 0)=b_0(x),
	\end{cases}
\end{align}
where we denote $b - {\bf B_0}$ by $b$ for convenience.

We now explain the main ideas of the proof of  Theorem \ref{thm}. It is recalled that  Himonas and Misiołek  \cite{nonuniform-2010-cmp-himonas} introduced a frequency decomposition to construct  approximate solutions of the two-dimensional incompressible Euler equations. Denote
 \begin{align}\label{1.1}
	u^{\pm1,\lambda}(x,t)=u^{l,\pm1,\lambda}(x,t)+u^{h,\pm1,\lambda}(x,t),
\end{align}
where
\begin{itemize}
	\item[(i)] the low-frequency components $u^{l,\pm1,\lambda}$ are  obtained by solving the Euler equations with  low-frequency initial data;
	\item[(ii)] the high-frequency components $u^{h,\pm1,\lambda}$ are explicitly constructed by using  oscillatory profiles, with frequency parameter $\lambda$ controlling the spatial concentration,
	\begin{align}\label{1.2}
		u^{h,\pm1,\lambda}(x,t)=\nabla^\perp \left(\lambda^{-\delta-s-1}\phi\left(\dfrac{x_1}{\lambda^\delta}\right)\phi\left(\dfrac{x_2}{\lambda^\delta}\right)\sin (\lambda x_2\mp t)\right).
	\end{align}
Here, $\lambda>0, \max\{1-s,0\}<\delta < 1$ and $\phi\in C_c^\infty(\mr)$ with $\text{supp}\,\phi \subset [-2,2]$ and $\phi(x) \equiv 1$ on $|x| < 1$.
\end{itemize}

It is required to modify the construction of the approximate solutions in the presence of the magnetic field. To the resistive MHD equations, due to the diffusive nature of the magnetic field evolution (versus the purely transport equations of the Euler equations ), we restrict high-frequency perturbations to the velocity field alone and  maintain low-frequency components for the magnetic field. More precisely, in the case $\mathbf{B_0} = \mathbf{0}$ (no background magnetic field) and $d=2$,  we set
\begin{itemize}
	\item the low-frequency pairs $(u^{l,\pm1,\lambda}, b^{\pm1,\lambda})$ by solving the resistive MHD equations \eqref{mhd} (equivalently, system \eqref{mhd-B} with $\mathbf{B_0} = \mathbf{0}$) with low-frequency initial data;
	\item the high-frequency velocity components $u^{h,\pm1,\lambda}$ by retain their profiles as in  \eqref{1.2}.
\end{itemize}

In the three-dimensional case $(d = 3)$, we modify  the high-frequency velocity components as
	\begin{align}\label{1.3}
	u^{h,\pm1,\lambda}(x,t)=\begin{pmatrix}
		\partial_{x_1}\\[1mm]
		-\partial_{x_2}\\[1mm]
		0
	\end{pmatrix} \left(\lambda^{-\delta-s-1}\phi\left(\dfrac{x_1}{\lambda^\delta}\right)\phi\left(\dfrac{x_2}{\lambda^\delta}\right)\sin (\lambda x_2\mp t)\phi(x_3)\right).
\end{align}

The error terms induced by the  high-frequency  $u^{h,\pm1,\lambda}$   can be written as
\begin{align}\label{1.4}
	E^{\pm1,\lambda}&=\partial_t u^{h,\pm1,\lambda}+u^{l,\pm1,\lambda}\cdot\nabla u^{h,\pm1,\lambda}+u^{h,\pm1,\lambda}\cdot\nabla u^{h,\pm1,\lambda}+u^{h,\pm1,\lambda}\cdot\nabla u^{l,\pm1,\lambda},\\\label{1.5}
	F^{\pm1,\lambda}&=u^{h,\pm1,\lambda}\cdot\nabla b^{\pm1,\lambda}-b^{\pm1,\lambda}\cdot\nabla u^{h,\pm1,\lambda},
\end{align}
respectively.

Direct estimation of \eqref{1.5} will lead to uncontrolled error terms. Our alternative approach consists of three key steps:
\begin{enumerate}
	\item {\bf Divergence reformulation:} We express $F^{\pm1,\lambda} =\diver\, \tilde{F}^{\pm1,\lambda}$ and establish estimates for $\tilde{F}^{\pm1,\lambda}$;
	\item {\bf Integration by parts:} In $L^2$ inner product computations, we transfer the divergence operator to $b^{\pm1,\lambda}$ via integration by parts;
	\item {\bf Diffusion cancellation:} The resulting terms are precisely canceled by exploiting the magnetic diffusion term $\Delta b^{\pm1,\lambda}$.
\end{enumerate}

For non-zero background fields $\mathbf{B_0}=(B_1,B_2) \neq \mathbf{0}$, we introduce the coordinate transformation
\begin{align*}
	x_2 \mapsto B_1x_2 - B_2x_1,
\end{align*}
which maintains all norm estimates while achieving exact cancellation of the dominant high-frequency linear term $\mathbf{B_0} \cdot \nabla u^{h,\pm 1,\lambda}$ appearing in the error analysis of $F^{\pm1,\lambda}$.

A key observation is that the directional derivative exhibits a better decay:
\begin{align*}
	\left\|\mathbf{B_0} \cdot \nabla u^{h,\pm 1,\lambda}\right\|_{L^2}\leq C\lambda^{-s-\delta} \ll \left\|u^{h,\pm 1,\lambda}\right\|_{L^2}\leq C\lambda^{-s}, \text{~for~}\lambda \to \infty,
\end{align*}
where $\delta>0$ represent the improved decay rate.

	 The paper is organized as follows. section \ref{sec:2} is the preliminary analytical framework, in which subsection \ref{sec:2.1} is about key Lemmas and subsection \ref{sec:2.2} is on the local well-posedness theory. The core  strategy of the proof is implemented in section \ref{sec:3} (zero background field case, $\mathbf{B_0} = \mathbf{0}$) and section  \ref{sec:4} (non-zero background field case, $\mathbf{B_0} \neq \mathbf{0}$), following the approach as follows: construction of approximate solutions in subsections \ref{sec:3.1} and \ref{sec:4.1}, precise estimation of approximation errors in subsections \ref{sec:3.2} and \ref{sec:4.2}, rigorous construction of exact solutions in subsections \ref{sec:3.3} and \ref{sec:4.3}, finishing the proof of Theorem in the case $\mathbf{B_0} = \mathbf{0}$ and $\mathbf{B_0} \neq \mathbf{0}$ in subsections \ref{sec:3.4} and \ref{sec:4.4}, respectively.

{\bf Notations:}
\begin{enumerate}
\item {\bf Function Spaces:} Throughout this work, let $X$ denote a Banach space equipped with norm $\|\cdot\|_X$. Since all function spaces considered here are defined on $\mathbb{R}^d$ for $d = 2,3$, we will suppress the domain $\mathbb{R}^d$ in our notation unless otherwise specified.
	\item {\bf Differential Operators:} For $x \in \mathbb{R}^2$, we define the perpendicular gradient operator as $\nabla^\perp := (\partial_{x_2}, -\partial_{x_1})$.
	\item {\bf Joint Norms:} Given functions $f, g \in X(\mathbb{R}^d)$, we define their joint norm by
\begin{align*}
	\|f(\cdot),g(\cdot)\|_{X(\mr^d)}=	\|f(\cdot)\|_{X(\mr^d)}+	\|g(\cdot)\|_{X(\mr^d)}
\end{align*}
adopting this concise notation for simplicity.
\end{enumerate}

\section{Preliminaries}\label{sec:2}
\subsection{Auxiliary Analysis Tools}\label{sec:2.1}
In this subsection, we introduce two key lemmas needed later. The first one is
\begin{lemma}\cite{nonuniform-2010-cmp-himonas}\label{phi-cmp-2010}
	Let $\sigma \geq 0, \delta\geq 0, a\in \mr$ and $\lambda\gg 1$. For any Schwartz function $\psi\in \mathcal{S}(\mr)$, it holds that
	\begin{align*}
		\lambda^{\delta/2}\left\|\psi\right\|_{L^2(\mathbb{R})}\leq\left\|\psi\left(\frac{\cdot}{\lambda^\delta}\right)\right\|_{H^\sigma(\mathbb{R})}\leq\lambda^{\delta/2}\left\|\psi\right\|_{H^\sigma(\mathbb{R})},
	\end{align*}
and \begin{align}\label{622}
	\left\|\psi\left(\frac{\cdot}{\lambda^\delta}\right)\cos(\lambda\cdot-a)\right\|_{H^\sigma(\mathbb{R})}\simeq\lambda^{\sigma+\delta/2}\|\psi\|_{L^2(\mathbb{R})}.
\end{align}
Moreover, \eqref{622}  holds true if  $\cos(\lambda \cdot-a)$ is replaced by $\sin(\lambda \cdot-a)$.
\end{lemma}

Before stating the second lemma, we present two systems as follows.
\begin{align}\label{mhd1}
	\begin{cases}
		\partial_t u + u \cdot \nabla u + \nabla p = b \cdot \nabla b+{\bf B_0}\cdot\nabla b, \\[1mm]
		\partial_t b - \Delta b + u \cdot \nabla b = b \cdot \nabla u+{\bf B_0}\cdot \nabla u, \\[1mm]
		\diver\, u = \diver\, b = 0, \\[1mm]
		u(x, 0) = u_0(x), \quad b(x, 0) = b_0(x),
	\end{cases}
\end{align}
and
\begin{align}\label{mhd2}
	\begin{cases}
		\partial_t u^E + u^E \cdot \nabla u^E + \nabla p^E = b^F \cdot \nabla b^F +{\bf B_0}\cdot\nabla b^F+ E, \\[1mm]
		\partial_t b^F - \Delta b^F + u^E \cdot \nabla b^F = b^F \cdot \nabla u^E +{\bf B_0}\cdot \nabla u^E + \text{div}\,F_1+F_2, \\[1mm]
		\diver\, u^E = \diver\, b^F = 0, \\[1mm]
		u^E(x, 0) = u_0(x), \quad b^F(x, 0) = b_0(x),
	\end{cases}
\end{align}
where ${\bf B_0}$ denotes an arbitrary constant vector in $\mathbb{R}^d$.

 The  second lemma can be stated as

\begin{lemma}\label{two-systems-E-F}
Suppose that  $(u^E,b^F)$ and $(u,b)$ are solutions of \eqref{mhd1} and \eqref{mhd2}  defined on  the time interval $[0,T]$, respectively.  Then, the differences  $\eta \define u - u^E$ and $\xi \define b - b^F$
satisfy
\begin{align}\nonumber
	\left(\|\eta(t)\|_{L^2}+\|\xi(t)\|_{L^2}\right)\leq&t\exp\left(2t\max_{0\leq \tau \leq t}\left(\|\nabla u^E(\tau)\|_{L^\infty} + \|\nabla b^F(\tau)\|_{L^\infty}+1\right)\right) \\\label{energy_estimate}
	&\times\left(\max_{0\leq \tau \leq t}\left(\|E(\tau)\|_{L^2}^2+\|F_1(\tau)\|_{L^2}^2+\|F_2(\tau)\|_{L^2}^2\right)\right)^{\frac{1}{2}},
\end{align}
for all $t\in[0,T].$
\end{lemma}
\begin{proof}
It follows from \eqref{mhd1} and \eqref{mhd2} that the differences $(\eta,\xi) = (u-u^E,b-b^F)$ satisfy
	\begin{align}\label{eta,xi-equations}
		\begin{cases}
			\partial_t \eta+u \cdot \nabla \eta+\eta \cdot\nabla u^E+\nabla \tilde{p}=b\cdot \nabla \xi+\xi\cdot\nabla b^F+{\bf B_0}\cdot\nabla \xi-E,\\[1mm]
			\partial_t \xi-\Delta \xi+u\cdot \nabla \xi +\eta \cdot\nabla b^F=b\cdot\nabla \eta +\xi\cdot\nabla u^E+{\bf B_0}\cdot\nabla \eta-\text{div}\,F_1-F_2,  \\[1mm]
			\diver\, u^E=\diver\, b^F=\diver\, u=\diver\, b=\diver\, \eta=\diver\, \xi =0, \\[1mm]
			\eta (x,0)=\xi (x,0)=0,
		\end{cases}
	\end{align}
where $\tilde{p}=p-p^E$.

Taking the $L^2$-inner products of \eqref{eta,xi-equations}$_1$ and \eqref{eta,xi-equations}$_2$ with $u$ and $b$ respectively yields
\begin{align}\label{energy_identity}
	\frac{1}{2}\frac{\mathrm{d}}{\mathrm{d}t}\left(\|\eta(t)\|_{L^2}^2 + \|\xi(t)\|_{L^2}^2\right) + \|\nabla \xi(t)\|_{L^2}^2 = \sum_{i=1}^{14} I_i,
\end{align}
	where
	\begin{align*}
		I_1&=-\int_{\mr^d}\left(u\cdot \nabla \eta\right) \cdot \eta\dd x,~~I_2=-\int_{\mr^d}\left(\eta \cdot\nabla u^E\right) \cdot \eta\dd x,\\
		I_3&=-\int_{\mr^d}\nabla \tilde{p} \cdot \eta\dd x,~~I_4=\int_{\mr^d}\left(b\cdot \nabla \xi\right)\cdot \eta\dd x,\\
		I_5&=\int_{\mr^d}\left(\xi\cdot\nabla b^F\right)\cdot\eta\dd x,~~I_6=\int_{\mr^d}{\bf B_0}\cdot\nabla\xi \cdot\eta\dd x,\\
		I_7&=-\int_{\mr^d}\left(u\cdot\nabla \xi\right)\cdot\xi \dd x,~~I_8=-\int_{\mr^d}\left(\eta\cdot\nabla b^F\right)\cdot\xi\dd x,\\
		I_9&=\int_{\mr^d}\left(b\cdot\nabla \eta \right)\cdot\xi\dd x,~~I_{10}=\int_{\mr^d}\left(\xi\cdot\nabla u^E\right)\cdot \xi\dd x,\\
		I_{11}&=\int_{\mr^d} {\bf B_0}\cdot\nabla \eta\cdot\xi\dd x,	~~I_{12}=-\int_{\mr^d}E\cdot\eta \dd x\\
	I_{13}&=-\int_{\mr^d}\diver\,F_1\cdot\xi\dd x,~~I_{14}=-\int_{\mr^d}F_2\cdot\xi\dd x.
	\end{align*}
Through integration by parts and exploiting the divergence-free properties of $u,u^E$ and $b,b^F$, it follows that
	\begin{align*}
		I_1=I_3=I_7=0, ~I_4+I_9=0,~ I_6+I_{11}=0.
	\end{align*}
An application of Hölder's inequality and Young's inequality yields
	\begin{align*}
				I_2+I_{10}&\leq \|\nabla u^E(t)\|_{L^\infty}\left(\|\eta(t)\|_{L^2}^2+\|\xi(t)\|_{L^2}^2\right),\\[2mm]
		I_5+I_8&\leq 2\|\nabla b^F(t)\|_{L^\infty}\|\eta(t)\|_{L^2}\|\xi(t)\|_{L^2},\\[2mm]
				I_{12}&\leq \|\eta(t)\|_{L^2}\|E(t)\|_{L^2}\leq \dfrac{1}{2}\|\eta(t)\|_{L^2}^2+\dfrac{1}{2}\|E(t)\|_{L^2}^2,\\[2mm]
				I_{13}&\leq \|\nabla\xi(t)\|_{L^2}\|F_1(t)\|_{L^2}\leq \dfrac{1}{2}\|\nabla\xi(t)\|_{L^2}^2+\dfrac{1}{2}\|F_1(t)\|_{L^2}^2,\\[2mm]
					I_{14}&\leq \|\xi(t)\|_{L^2}\|F_2(t)\|_{L^2}\leq \dfrac{1}{2}\|\xi(t)\|_{L^2}^2+\dfrac{1}{2}\|F_2(t)\|_{L^2}^2.
	\end{align*}
The energy identity \eqref{energy_identity} can be reformulated as
\begin{align}\label{energy_reformulated}
	&\frac{\mathrm{d}}{\mathrm{d}t}\left(\|\eta(t)\|_{L^2}^2 + \|\xi(t)\|_{L^2}^2\right) \nonumber \\
	&\qquad\leq 4\left(\|\nabla u^E(t)\|_{L^\infty} + \|\nabla b^F(t)\|_{L^\infty}+1\right)
	\left(\|\eta(t)\|_{L^2}^2 + \|\xi(t)\|_{L^2}^2\right) \nonumber \\
	&\quad\qquad + \dfrac{1}{2}\|E(t)\|_{L^2}^2+\dfrac{1}{2}\|F_1(t)\|_{L^2}^2+\dfrac{1}{2}\|F_2(t)\|_{L^2}^2.
\end{align}
An application of Grönwall's inequality then yields
\begin{align*}
	&\left(\|\eta(t)\|_{L^2}^2+\|\xi(t)\|_{L^2}^2\right)\\
	&\leq \dfrac{1}{2}\exp\left(4\int_{0}^t\left(\|\nabla u^E(\tau)\|_{L^\infty} + \|\nabla b^F(\tau)\|_{L^\infty}+1\right)\dd\tau\right)\\
	&~~~~~~\int_{0}^{t}\|E(s)\|_{L^2}^2+\|F_1(s)\|_{L^2}^2+\|F_2(s)\|_{L^2}^2\dd s,
\end{align*}
which implies \eqref{energy_estimate} and the proof of the lemma is finished.
\end{proof}
\subsection{Local well-posedness}\label{sec:2.2}
In this subsection, we  obtain some a priori estimates, based on which  the local well-posedness can be rigorously demonstrated by using standard approximation methods (see \cite{majda-incompressible flow,kato-1988-CommutatorEstimatesEuler,sermange-1983}).
\begin{lemma}\label{local-well-pri}
	Let $m>\frac{d}{2}+1, T>0$ and $u_0,b_0\in H^m(\mr^d)$. If $(u,b)\in C\left([0,T],H^m(\mr^d)\right)$ is the unique solution to the Cauchy problem \eqref{mhd-B}, it holds that
\begin{align*}
	\|u(t),b(t)\|_{H^m}\leq \dfrac{\|u_0,b_0\|_{H^m}}{1-Ct\|u_0,b_0\|_{H^m}},
\end{align*}
where $0\leq t\leq T<C^{-1}\|u_0,b_0\|_{H^m}^{-1}$ and $C>0$ is some constant.
\end{lemma}
\begin{proof}
Applying the operator $\Lambda^m$ to both sides of \eqref{mhd-B}$_1$ and \eqref{mhd-B}$_2$, taking the inner product of the resulting equations with $\Lambda^m u$ and $\Lambda^m b$, respectively, and summing them, we obtain
	\begin{align}\nonumber
	&\frac{1}{2}	\frac{\dd}{\dd t}\left(\|\Lambda^m u\|_{L^2}^2+\|\Lambda^m b\|_{L^2}^2\right)+\|\Lambda^{m+1} b\|_{L^2}^2=\int_{\mr^d}\Lambda^m\nabla p\cdot\Lambda^m u\dd x\\\nonumber
	&~~~~~+\int_{\mr^d} \Lambda^m\left(b\cdot\nabla b-u\cdot\nabla u\right)\cdot \Lambda^m u\dd x+\int_{\mr^d} \Lambda^m\left(b\cdot\nabla u-u\cdot \nabla b\right)\cdot \Lambda^m b\dd x\\\label{2.1}
	&~~~~~+\int_{\mr^d} \Lambda^m\left({\bf B_0}\cdot\nabla b\right)\cdot \Lambda^m u\dd x+\int_{\mr^d} \Lambda^m\left({\bf B_0}\cdot\nabla u\right)\cdot \Lambda^m b\dd x.
\end{align}
Using integration by parts and the divergence free condition
$\text{div}u=\text{div}b=0$ leads to
\begin{align*}
	\int_{\mr^d}\Lambda^m\nabla p\cdot\Lambda^m u\dd x=\int_{\mr^d}\Lambda^m p\cdot\Lambda^m \diver\, u \dd x=0,
\end{align*}
and
\begin{align*}
	\int_{\mr^d}u\cdot\nabla (\Lambda^m u) \cdot\Lambda^m u\dd x=\int_{\mr^d}b\cdot\nabla (\Lambda^m b) \cdot\Lambda^m b\dd x=0,\\
		\int_{\mr^d} b\cdot\nabla \left(\Lambda^m b\right) \cdot \Lambda^m u \dd x+	\int_{\mr^d} b\cdot\nabla \left(\Lambda^m u\right) \cdot \Lambda^m b \dd x=0,\\
		\int_{\mr^d} \Lambda^m\left({\bf B_0}\cdot\nabla b\right)\cdot \Lambda^m u\dd x+\int_{\mr^d} \Lambda^m\left({\bf B_0}\cdot\nabla u\right)\cdot \Lambda^m b\dd x=0.
\end{align*}
The right-hand side of \eqref{2.1} can be estimated as
\begin{align}\nonumber
	&\frac{1}{2}	\frac{\dd}{\dd t}\left(\|\Lambda^m u\|_{L^2}^2+\|\Lambda^m b\|_{L^2}^2\right)+\|\Lambda^{m+1} b\|_{L^2}^2\\\label{2.2}
	&~~~~~~\leq C\left(\|\Lambda^m u\|_{L^2}^2+\|\Lambda^m b\|_{L^2}^2\right)\left(\|\nabla u\|_{L^\infty}+\|\nabla b\|_{L^\infty}\right).
\end{align}
Standard  $L^2$ energy estimate is
\begin{align}\label{2.3}
	\frac{1}{2}	\frac{\dd}{\dd t}\left(\| u\|_{L^2}^2+\| b\|_{L^2}^2\right)+\|\nabla b\|_{L^2}^2=0.
\end{align}
Combining \eqref{2.2} with \eqref{2.3} gives
\begin{align*}
	&\frac{1}{2}	\frac{\dd}{\dd t}\left(\|u\|_{H^m}^2+\| b\|_{H^m}^2\right)+\|\nabla b\|_{H^m}^2\\
		&~~\leq C\left(\|u\|_{H^m}^2+\| b\|_{H^m}^2\right)\left(\|\nabla u\|_{L^\infty}+\|\nabla b\|_{L^\infty}\right)\leq C\left(\|u\|_{H^m}^2+\| b\|_{H^m}^2\right)^\frac{3}{2}.
\end{align*}
Using the  Gronwall's inequality yields
\begin{align*}
	\|u(t)\|_{H^m}^2+\| b(t)\|_{H^m}^2\leq \dfrac{4\left(\|u(0)\|_{H^m}^2+\| b(0)\|_{H^m}^2\right)}{\left(2-Ct\sqrt{\|u(0)\|_{H^m}^2+\| b(0)\|_{H^m}^2}\right)^2}.
\end{align*}
The proof of the lemma is complete.
\end{proof}
\section{The Case ${\bf B_0}={\bf 0}$}\label{sec:3}
~~~~When the constant background magnetic field is  zero (${\bf B_0} = \mathbf{0}$),  the equations \eqref{mhd-B} are equivalent to  \eqref{mhd}.  In this section, we prove our main result Theorem \ref{thm} in the  case ${\bf B_0} = \mathbf{0}$.

\subsection{Approximate Solutions Scheme}\label{sec:3.1}
We begin by constructing two sequences of approximate solutions $b^{w,\lambda}$ and $u^{w,\lambda}$ with $w=\pm 1$, where $b^{w,\lambda}$ contains  low frequencies only and $u^{w,\lambda}$ contain both high and low frequencies. More precisely, we set
\begin{align}\label{u-low+high}
	u^{w,\lambda}(x,t)=u^{l,w,\lambda}(x,t)+u^{h,w,\lambda}(x,t).
\end{align}
The high-frequency components $u^{h,w,\lambda}$  are constructed as
\begin{align}\label{u-high}
	u^{h,w,\lambda}(x,t)=\begin{cases}
		\nabla^\perp \phi^{h,w,\lambda}(x,t)=(\partial_{x_2}\phi^{h,w,\lambda}(x,t),-\partial_{x_1}\phi^{h,w,\lambda}(x,t)), &~d=2,\\[3mm]
		\left(\nabla^\perp \phi^{h,w,\lambda}(x,t),0 \right)=(\partial_{x_2}\phi^{h,w,\lambda}(x,t),-\partial_{x_1}\phi^{h,w,\lambda}(x,t),0), &~d=3.
	\end{cases}
\end{align}
In \eqref{u-high}, the function $\phi^{h,w,\lambda}(x,t)$ is defined as
\begin{align}\label{phi-high}
	\phi^{h,w,\lambda}(x,t)=\begin{cases}
		\lambda^{-s-\delta-1}\phi(\dfrac{x_1}{\lambda^\delta})\phi(\dfrac{x_2}{\lambda^\delta})\sin (\lambda x_2-wt),&~d=2,\\[3mm]
		\lambda^{-s-\delta-1}\phi(\dfrac{x_1}{\lambda^\delta})\phi(\dfrac{x_2}{\lambda^\delta})\sin (\lambda x_2-wt)\phi(x_3),&~d=3,
	\end{cases}
\end{align}
where $s > 0, \lambda>0, w=\pm 1$,  the parameter $\delta > 0$ is to be specified later and  the smooth function $\phi \in C_c^\infty(\mathbb{R})$ satisfies $\text{supp}\,\phi \subset [-2,2]$ and $\phi(x) \equiv 1$ on $|x| < 1$.

To construct the low-frequency components, we choose smooth function $\Phi_1,\Phi_2\in C_c^\infty(\mr)$ such that $\Phi_1'=\Phi_2=1$ on the support of $\phi$. Then the low-frequency components $u^{l,w,\lambda}$  and $b^{w,\lambda}$ are constructed by solving
\begin{align}\label{mhd-low}
	\begin{cases}
		\partial_t u^{l,w,\lambda}+u^{l,w,\lambda} \cdot \nabla u^{l,w,\lambda}+\nabla p^{l,w,\lambda}=b^{w,\lambda}\cdot \nabla b^{w,\lambda},\\[1mm]
		\partial_t b^{w,\lambda}-\Delta b^{w,\lambda}+u^{l,w,\lambda}\cdot \nabla b^{w,\lambda}=b^{w,\lambda}\cdot\nabla u^{l,w,\lambda},  \\[1mm]
		\diver\, u^{l,w,\lambda}=\diver\, b^{w,\lambda}=0,
	\end{cases}
\end{align}
with the initial data
\begin{align*}
	u^{l,w,\lambda}(x, 0)=b^{w,\lambda}(x,0)=\begin{cases}
	\nabla^\perp\phi^{l,w,\lambda}(x),&~d=2,\\[2mm]
	\left(	\nabla^\perp\phi^{l,w,\lambda}(x),0\right),&~d=3,
\end{cases}
\end{align*}
where the stream function $\phi^{l,w,\lambda}(x)$ is given by
\begin{align}\label{phi-low-Phi-cutoff}
	\phi^{l,w,\lambda}(x)=\begin{cases}
		-w\lambda^{-1+\delta}\Phi_1(\dfrac{x_1}{\lambda^\delta})\Phi_2(\dfrac{x_2}{\lambda^\delta}),&~d=2,\\[3mm]
		-w\lambda^{-1+\delta}\Phi_1(\dfrac{x_1}{\lambda^\delta})\Phi_2(\dfrac{x_2}{\lambda^\delta})\phi(x_3),&~d=3.
	\end{cases}
\end{align}

Direct computation shows that the initial data
\begin{align}
u^{l,w,\lambda}(x,0) = b^{w,\lambda}(x,0) = \begin{cases}
	\begin{pmatrix}
		-w\lambda^{-1}\Phi_1\left(\dfrac{x_1}{\lambda^\delta}\right)\Phi_2'\left(\dfrac{x_2}{\lambda^\delta}\right) \\[4mm]
		w\lambda^{-1}\Phi_1'\left(\dfrac{x_1}{\lambda^\delta}\right)\Phi_2\left(\dfrac{x_2}{\lambda^\delta}\right)
	\end{pmatrix},&~d=2,\\[12mm]
	\begin{pmatrix}
		-w\lambda^{-1}\Phi_1\left(\dfrac{x_1}{\lambda^\delta}\right)\Phi_2'\left(\dfrac{x_2}{\lambda^\delta}\right)\phi(x_3) \\[4mm]
		w\lambda^{-1}\Phi_1'\left(\dfrac{x_1}{\lambda^\delta}\right)\Phi_2\left(\dfrac{x_2}{\lambda^\delta}\right)\phi(x_3)\\[4mm]
		0
	\end{pmatrix},&~d=3,
\end{cases}
\end{align}
Applying Lemma \ref{local-well-pri}  and Lemma \ref{phi-cmp-2010}, we derive, for all $m > \frac{d}{2} + 1$, the uniform bounds
\begin{align}\label{u,b-low-t-local-bound}
	\|u^{l,w,\lambda}(t)\|_{H^m} + \|b^{w,\lambda}(t)\|_{H^{m}}
	&\leq C\left(\|u^{l,w,\lambda}(0)\|_{H^m} + \|b^{w,\lambda}(0)\|_{H^m}\right) \nonumber \\
	&= C\|\nabla^\perp\phi^{l,w,\lambda}\|_{H^m} \leq C\lambda^{-1+\delta}
\end{align}
 hold for $0 \leq t \leq T$, where $T$ denotes the existence time as stated in Lemma~\ref{local-well-pri}. Without loss of generality, we may assume $T < 1$ when we take  $\min\{T,1\}$.

It follows from \eqref{u-high} that
\begin{align}\label{u{hw,lambda}-two-pmatrix}
	u^{h,w,\lambda}=\begin{pmatrix}
		\lambda^{-s-\delta}\phi(\dfrac{x_1}{\lambda^\delta})\phi(\dfrac{x_2}{\lambda^\delta})\cos (\lambda x_2-wt)+\lambda^{-s-1-2\delta}\phi(\dfrac{x_1}{\lambda^\delta})\phi'(\dfrac{x_2}{\lambda^\delta})\sin (\lambda x_2-wt)\\[4mm]
		- \lambda^{-s-1-2\delta}\phi'(\dfrac{x_1}{\lambda^\delta})\phi(\dfrac{x_2}{\lambda^\delta})\sin (\lambda x_2-wt)
	\end{pmatrix}
\end{align}
 for $d=2$, and
\begin{align}\label{u{hw,lambda}-two-pmatrix1}
	u^{h,w,\lambda}=\begin{pmatrix}
	\left(	\lambda^{-s-\delta}\phi(\dfrac{x_1}{\lambda^\delta})\phi(\dfrac{x_2}{\lambda^\delta})\cos (\lambda x_2-wt)+\lambda^{-s-1-2\delta}\phi(\dfrac{x_1}{\lambda^\delta})\phi'(\dfrac{x_2}{\lambda^\delta})\sin (\lambda x_2-wt)\right)\phi(x_3)\\[4mm]
		- \lambda^{-s-1-2\delta}\phi'(\dfrac{x_1}{\lambda^\delta})\phi(\dfrac{x_2}{\lambda^\delta})\sin (\lambda x_2-wt)\phi(x_3)\\[4mm]
		0
	\end{pmatrix}
\end{align}
for $d=3$.

Thanks to Lemma \ref{phi-cmp-2010}, for any $r\geq 0$, one has
\begin{align}\nonumber
&\left\|	\lambda^{-s-\delta}\phi(\dfrac{x_1}{\lambda^\delta})\phi(\dfrac{x_2}{\lambda^\delta})\cos (\lambda x_2-wt)\right\|_{H^r(\mr^2)}\\\nonumber
&\leq \lambda^{-s-\delta}\left\|	\phi(\dfrac{x_1}{\lambda^\delta})\right\|_{H^r(\mr)}\left\|	\phi(\dfrac{x_2}{\lambda^\delta})\cos (\lambda x_2-wt)\right\|_{H^r(\mr)}\\\label{3.1}
&\leq C\lambda^{-s-\delta}\lambda^{r+\delta}=C\lambda^{r-s},
\end{align}
and \begin{align*}
	&\left\|	\lambda^{-s-\delta}\phi(\dfrac{x_1}{\lambda^\delta})\phi(\dfrac{x_2}{\lambda^\delta})\cos (\lambda x_2-wt)\phi(x_3)\right\|_{H^r(\mr^3)}\\
	&\leq \left\|	\lambda^{-s-\delta}\phi(\dfrac{x_1}{\lambda^\delta})\phi(\dfrac{x_2}{\lambda^\delta})\cos (\lambda x_2-wt)\right\|_{H^r(\mr^2)}\|\phi(x_3)\|_{H^r(\mr)}\\
	&\leq C\lambda^{r-s},
\end{align*}
It follows that
\begin{align}\label{u-high-bound}
	\|u^{h,w,\lambda}(t)\|_{H^r}\leq C\lambda^{r-s},\quad \|u^{h,w,\lambda}(t)\|_{L^\infty}\leq C\lambda^{-s-\delta},\quad \|\nabla u^{h,w,\lambda}(t)\|_{L^\infty}\leq C\lambda^{-s+1-\delta}.
\end{align}
Using \eqref{u,b-low-t-local-bound} and \eqref{u-high-bound}, we have
\begin{align}\nonumber
	&\|\nabla u^{w,\lambda}(t)\|_{L^\infty}+\|\nabla b^{w,\lambda}(t)\|_{L^\infty}\\\nonumber
	&\leq 	\|\nabla u^{l,w,\lambda}(t)\|_{L^\infty}+\|\nabla u^{h,w,\lambda}(t)\|_{L^\infty}+\|\nabla b^{w,\lambda}(t)\|_{L^\infty}\\\nonumber
		&\leq \|\nabla u^{h,w,\lambda}(t)\|_{L^\infty}+	C\left(\| u^{l,w,\lambda}(t)\|_{H^m}+\| b^{w,\lambda}(t)\|_{H^m}\right)\\\label{nabla-u-b-t-bound}
	&\leq C\left(\lambda^{-1+\delta}+\lambda^{-s+1-\delta}\right),
\end{align}
where $m>\frac{d}{2}+1$ and $0\leq t\leq T$.

\subsection{Error Estimates for the Approximate Solutions}\label{sec:3.2}
 In view of \eqref{u-low+high} and \eqref{mhd-low},  the approximate solutions $(u^{w,\lambda},b^{w,\lambda})$ solve
\begin{align}\label{mhd-low+high}
	\begin{cases}
		\partial_t u^{w,\lambda}+u^{w,\lambda} \cdot \nabla u^{w,\lambda}+\nabla p^{w,\lambda}=b^{w,\lambda}\cdot \nabla b^{w,\lambda}+E^{w,\lambda},\\[1mm]
		\partial_t b^{w,\lambda}-\Delta b^{w,\lambda}+u^{w,\lambda}\cdot \nabla b^{w,\lambda}=b^{w,\lambda}\cdot\nabla u^{w,\lambda}+F^{w,\lambda},  \\[1mm]
		\diver\, u^{w,\lambda}=\diver\, b^{w,\lambda}=0,
	\end{cases}
\end{align}
where
\begin{align}\label{E{w,lambda}}
	E^{w,\lambda}&=\partial_t u^{h,w,\lambda}+u^{l,w,\lambda}\cdot\nabla u^{h,w,\lambda}+u^{h,w,\lambda}\cdot\nabla u^{h,w,\lambda}+u^{h,w,\lambda}\cdot\nabla u^{l,w,\lambda},\\\nonumber
	F^{w,\lambda}&=u^{h,w,\lambda}\cdot\nabla b^{w,\lambda}-b^{w,\lambda}\cdot\nabla u^{h,w,\lambda}\\\label{F{w,lambda}}
	&=\text{div}
	\left(u^{h,w,\lambda}\otimes b^{w,\lambda}-b^{w,\lambda}\otimes u^{h,w,\lambda}\right)\define\text{div} \tilde{F}^{w,\lambda}.
\end{align}

Moreover, the initial data are
\begin{align*}
			u^{w,\lambda}(x, 0)&=\begin{cases}
				\nabla^\perp\left( \phi^{l,w,\lambda}(x)+\phi^{h,w,\lambda}(x,0)\right),&~d=2,\\[2mm]
				\left(	\nabla^\perp\left( \phi^{l,w,\lambda}(x)+\phi^{h,w,\lambda}(x,0)\right),0\right),&~d=3,
			\end{cases}\\
			b^{w,\lambda}(x, 0)&=\begin{cases}
				\nabla^\perp\phi^{l,w,\lambda}(x),&~d=2,\\[2mm]
				\left(	\nabla^\perp\phi^{l,w,\lambda}(x),0\right),&~d=3.
			\end{cases}
\end{align*}

 To handle the error terms $E^{w,\lambda}$  and $F^{w,\lambda}$ appearing in \eqref{mhd-low+high}, which are defined as in \eqref{E{w,lambda}} and \eqref{F{w,lambda}} respectively, we prove
\begin{lemma}\label{E,F-error-estimate}
	 For any $s>0, \delta>0, \lambda\gg 1$ and $0\leq t\leq T$, it holds that
	\begin{align}\label{E-F-error-lambda-decay}
		\|E^{w,\lambda}(t)\|_{L^2}\leq C\lambda^{-\sigma_{s,\delta}},\quad	\|\tilde{F}^{w,\lambda}(t)\|_{L^2}\leq C\lambda^{-s-1},
	\end{align}
	where \begin{align}\label{sigma-s-delta}
		\sigma_{s,\delta}\define\min\{s+1-\delta,2s-1+\delta\}.
	\end{align}
\end{lemma}
\begin{proof}
	The estimate of the term $\tilde{F}^{w,\lambda}(t)$ is direct. It follows from   \eqref{u,b-low-t-local-bound} and \eqref{u-high-bound} that
	\begin{align*}
		\|\tilde{F}^{w,\lambda}(t)\|_{L^2}&\leq \|u^{h,w,\lambda}\|_{L^\infty}\|b^{w,\lambda}\|_{L^2}\leq C\lambda^{-s-\delta}\lambda^{-1+\delta}=C\lambda^{-s-1}.
	\end{align*}

To estimate the term $E^{w,\lambda}$, we first consider the case  $d=2$. Use \eqref{u{hw,lambda}-two-pmatrix} to obtain
\begin{align*}
	\partial_t u^{h,w,\lambda}=\begin{pmatrix}
		w	\lambda^{-s-\delta}\phi(\dfrac{x_1}{\lambda^\delta})\phi(\dfrac{x_2}{\lambda^\delta})\sin (\lambda x_2-wt)-w\lambda^{-s-1-2\delta}\phi(\dfrac{x_1}{\lambda^\delta})\phi'(\dfrac{x_2}{\lambda^\delta})\cos (\lambda x_2-wt)\\[4mm]
		w\lambda^{-s-1-2\delta}\phi'(\dfrac{x_1}{\lambda^\delta})\phi(\dfrac{x_2}{\lambda^\delta})\cos (\lambda x_2-wt)
	\end{pmatrix}.
\end{align*}
Since the first component in $\left(\partial_t u^{h,w,\lambda}\right)_1$ carries a larger exponent of $\lambda$,  we employ the cut-off functions $\Phi_1$ and $\Phi_2$ as in \eqref{phi-low-Phi-cutoff} to yield
\begin{align*}
	&w	\lambda^{-s-\delta}\phi(\dfrac{x_1}{\lambda^\delta})\phi(\dfrac{x_2}{\lambda^\delta})\sin (\lambda x_2-wt)\\
	&~~~~=w	\lambda^{-1}\Phi_1'(\dfrac{x_1}{\lambda^\delta})\Phi_2(\dfrac{x_2}{\lambda^\delta})\lambda^{-s+1-\delta}\phi(\dfrac{x_1}{\lambda^\delta})\phi(\dfrac{x_2}{\lambda^\delta})\sin (\lambda x_2-wt)\\
	&~~~~=u_2^{l,w,\lambda}(x,0)\lambda^{-s+1-\delta}\phi(\dfrac{x_1}{\lambda^\delta})\phi(\dfrac{x_2}{\lambda^\delta})\sin (\lambda x_2-wt).
\end{align*}
Consequently, we have
\begin{align*}
	&\left(\partial_t u^{h,w,\lambda}+u^{l,w,\lambda}\cdot\nabla u^{h,w,\lambda}\right)_1(x,t)\\
	&=w	\lambda^{-s-\delta}\phi(\dfrac{x_1}{\lambda^\delta})\phi(\dfrac{x_2}{\lambda^\delta})\sin (\lambda x_2-wt)-w\lambda^{-s-1-2\delta}\phi(\dfrac{x_1}{\lambda^\delta})\phi'(\dfrac{x_2}{\lambda^\delta})\cos (\lambda x_2-wt)\\
	&~~+u^{l,w,\lambda}\cdot \nabla \left(	\lambda^{-s-\delta}\phi(\dfrac{x_1}{\lambda^\delta})\phi(\dfrac{x_2}{\lambda^\delta})\cos (\lambda x_2-wt)+\lambda^{-s-1-2\delta}\phi(\dfrac{x_1}{\lambda^\delta})\phi'(\dfrac{x_2}{\lambda^\delta})\sin (\lambda x_2-wt)\right)\\
	&=\lambda^{-s+1-\delta}\phi(\dfrac{x_1}{\lambda^\delta})\phi(\dfrac{x_2}{\lambda^\delta})\sin (\lambda x_2-wt)\left(u_2^{l,w,\lambda}(x,0)-u_2^{l,w,\lambda}(x,t)\right)\\
	&~~-w\lambda^{-s-1-2\delta}\phi(\dfrac{x_1}{\lambda^\delta})\phi'(\dfrac{x_2}{\lambda^\delta})\cos (\lambda x_2-wt)\\
	&~~+ 2\lambda^{-s-2\delta}u_{2}^{l,w,\lambda}(x,t)\phi(\dfrac{x_1}{\lambda^\delta})\phi'(\dfrac{x_2}{\lambda^\delta})\cos (\lambda x_2-wt)\\
	&~~+\lambda^{-s-1-3\delta}u_{2}^{l,w,\lambda}(x,t)\phi(\dfrac{x_1}{\lambda^\delta})\phi''(\dfrac{x_2}{\lambda^\delta})\sin (\lambda x_2-wt)\\
	&~~+\lambda^{-s-2\delta}u_{1}^{l,w,\lambda}(x,t)\phi'(\dfrac{x_1}{\lambda^\delta})\phi(\dfrac{x_2}{\lambda^\delta})\cos (\lambda x_2-wt)\\
	&~~+\lambda^{-s-1-3\delta}u_{1}^{l,w,\lambda}(x,t)\phi'(\dfrac{x_1}{\lambda^\delta})\phi'(\dfrac{x_2}{\lambda^\delta})\cos (\lambda x_2-wt)\\
	&\define I_1+I_2+I_3+I_4+I_5+I_6.
\end{align*}

Using \eqref{u,b-low-t-local-bound} leads to
\begin{align*}
	&\left\|u_2^{l,w,\lambda}(x,0)-u_2^{l,w,\lambda}(x,t)\right\|_{L^2}\\
	&\leq t\left\| \partial_t u_2^{l,w,\lambda}(x,\tau)\right\|_{L^2}\\
	&\leq T\left\|\mathbb{P}\left(b^{w,\lambda}(\tau)\cdot\nabla b_2^{w,\lambda}(\tau)-u^{l,w,\lambda}(\tau)\cdot\nabla u^{l,w,\lambda}_2(\tau)\right)\right\|_{L^2}\\
	&\leq C\|\nabla (b^{w,\lambda}_2(\tau),u_2^{l,w,\lambda}(\tau))\|_{L^2}\|b^{w,\lambda}_2(\tau),u_2^{l,w,\lambda}(\tau)\|_{L^\infty}\leq C\lambda^{-2+2\delta}.
\end{align*}
The term $I_1$ is then estimated as
\begin{align*}
	\|I_1\|_{L^2}&\leq \lambda^{-s+1-\delta}\left\|\phi(\dfrac{x_1}{\lambda^\delta})\phi(\dfrac{x_2}{\lambda^\delta})\sin (\lambda x_2-wt)\right\|_{L^\infty}\left\|u_2^{l,w,\lambda}(x,0)-u_2^{l,w,\lambda}(x,t)\right\|_{L^2}\\
	&\leq C\lambda^{-s+1-\delta}\lambda^{-2+2\delta}=C\lambda^{-s-1+\delta}
\end{align*}
Applying Lemma \ref{phi-cmp-2010} yields
\begin{align*}
	\|I_2\|_{L^2}&\leq \lambda^{-s-1-2\delta}\left\|\phi(\dfrac{x_1}{\lambda^\delta})\right\|_{L^2(\mr)}\left\|\phi'(\dfrac{x_2}{\lambda^\delta})\cos (\lambda x_2-wt)\right\|_{L^2(\mr)}\leq C\lambda^{-s-1-\delta}.
\end{align*}
Applying \eqref{u,b-low-t-local-bound} yields
\begin{align*}
	\|I_3\|_{L^2}&\leq 2\lambda^{-s-2\delta}\|u^{l,w,\lambda}_2(t)\|_{L^2(\mr^2)}\left\|\phi(\dfrac{x_1}{\lambda^\delta})\right\|_{L^\infty(\mr)}\left\|\phi'(\dfrac{x_2}{\lambda^\delta})\cos (\lambda x_2-wt)\right\|_{L^\infty(\mr)}\\
	&\leq C\lambda^{-s-2\delta}\lambda^{-1+\delta}\leq C\lambda^{-s-1-\delta}.
\end{align*}
Similarly, we can obtain
\begin{align*}
	\|I_4\|_{L^2}&\leq C \lambda^{-s-1-3\delta}\left\|u_{2}^{l,w,\lambda}(t)\right\|_{L^2}\left\|\phi(\dfrac{x_1}{\lambda^\delta})\right\|_{L^\infty}\left\|\phi''(\dfrac{x_2}{\lambda^\delta})\sin (\lambda x_2-wt)\right\|_{L^\infty}\leq C\lambda^{-s-2-2\delta},\\
	\|I_5\|_{L^2}&\leq C\lambda^{-s-2\delta}\left\|u_{1}^{l,w,\lambda}(t)\right\|_{L^2}\left\|\phi'(\dfrac{x_1}{\lambda^\delta})\right\|_{L^\infty}\left\|\phi(\dfrac{x_2}{\lambda^\delta})\cos (\lambda x_2-wt)\right\|_{L^\infty}\leq C\lambda^{-s-1-\delta},\\
		\|I_6\|_{L^2}&\leq C\lambda^{-s-1-3\delta}\left\|u_{1}^{l,w,\lambda}(t)\right\|_{L^2}\left\|\phi'(\dfrac{x_1}{\lambda^\delta})\right\|_{L^\infty}\left\|\phi'(\dfrac{x_2}{\lambda^\delta})\cos (\lambda x_2-wt)\right\|_{L^\infty}\leq C\lambda^{-s-2-2\delta}.
\end{align*}
We are now in position to establish
\begin{align*}
	\left\|\left(\partial_t u^{h,w,\lambda}+u^{l,w,\lambda}\cdot\nabla u^{h,w,\lambda}\right)_1(x,t)\right\|_{L^2}\leq C\lambda^{-s-1+\delta}.
\end{align*}

For the second component of $\partial_t u^{h,w,\lambda}+u^{l,w,\lambda}\cdot\nabla u^{h,w,\lambda}$, we rewrite it as
\begin{align*}
	&\left(\partial_t u^{h,w,\lambda}+u^{l,w,\lambda}\cdot\nabla u^{h,w,\lambda}\right)_2(x,t)\\
	&=w	\lambda^{-s-1-2\delta}\phi'(\dfrac{x_1}{\lambda^\delta})\phi(\dfrac{x_2}{\lambda^\delta})\cos (\lambda x_2-wt)\\
	&-\lambda^{-s-1-3\delta}u_{1}^{l,w,\lambda}(x,t)\phi^{\prime\prime}\left(\frac{x_{1}}{\lambda^{\delta}}\right)\phi\left(\frac{x_{2}}{\lambda^{\delta}}\right)\sin(\lambda x_{2}-\omega t) \\
	&-\lambda^{-s-2\delta}u_{2}^{l,w,\lambda}(x,t)\phi^{\prime}\left(\frac{x_1}{\lambda^\delta}\right)\phi\left(\frac{x_2}{\lambda^\delta}\right)\cos(\lambda x_2-\omega t) \\
	&-\lambda^{-s-1-3\delta}u_{2}^{l,w,\lambda}(x,t)\phi^{\prime}\left(\frac{x_1}{\lambda^\delta}\right)\phi^{\prime}\left(\frac{x_2}{\lambda^\delta}\right)\sin(\lambda x_2-\omega t).
\end{align*}
Similar to the estimates on the first component $\partial_t u^{h,w,\lambda}+u^{l,w,\lambda}\cdot\nabla u^{h,w,\lambda}$, it holds that
\begin{align*}
	&\left\|\left(\partial_t u^{h,w,\lambda}+u^{l,w,\lambda}\cdot\nabla u^{h,w,\lambda}\right)_2(t)\right\|_{L^2(\mr^2)}\\
	&\leq C\lambda^{-s-1-\delta}+C\lambda^{-1+\delta}\left(\lambda^{-s-1-3\delta}+\lambda^{-s-2\delta}+\lambda^{-s-1-3\delta}\right)\\
	&\leq C\lambda^{-s-1-\delta}.
\end{align*}
Consequently, it follows that\begin{align}\nonumber
	&\left\|\left(\partial_t u^{h,w,\lambda}+u^{l,w,\lambda}\cdot\nabla u^{h,w,\lambda}\right)(t)\right\|_{L^2(\mr^2)}\\\nonumber
	&~~\leq \left\|\left(\partial_t u^{h,w,\lambda}+u^{l,w,\lambda}\cdot\nabla u^{h,w,\lambda}\right)_1(t)\right\|_{L^2(\mr^2)}\\\nonumber
	&~~~~~~+\left\|\left(\partial_t u^{h,w,\lambda}+u^{l,w,\lambda}\cdot\nabla u^{h,w,\lambda}\right)_2(t)\right\|_{L^2(\mr^2)}\\\label{E{w,lambda,1}}
	&~~\leq C\lambda^{-s-1+\delta}.
\end{align}

The remaining terms $u^{h,w,\lambda}\cdot\nabla u^{h,w,\lambda}$ and $u^{h,w,\lambda}\cdot\nabla u^{l,w,\lambda}$ in $E^{w,\lambda}$ can be estimated in a similar way,  which can be stated as
\begin{align}\label{E{w,lambda,2}}
	\left\|\left(u^{h,w,\lambda}\cdot\nabla u^{h,w,\lambda}\right)(t)\right\|_{L^2}\leq 	\left\|u^{h,w,\lambda}\right\|_{L^\infty}	\left\|\nabla u^{h,w,\lambda}\right\|_{L^2}\leq C\lambda^{-s-\delta}\lambda^{1-s}\leq C\lambda^{-2s+1-\delta},
\end{align}
and
\begin{align}\label{E{w,lambda,3}}
	\left\|\left(u^{h,w,\lambda}\cdot\nabla u^{l,w,\lambda}\right)(t)\right\|_{L^2}\leq \left\|u^{h,w,\lambda}\right\|_{L^2}\left\|\nabla u^{l,w,\lambda}\right\|_{L^2}\leq C\lambda^{-s-\delta}\lambda^{-1+\delta}\leq C\lambda^{-s-1}.
\end{align}
Thanks to \eqref{E{w,lambda,1}}, \eqref{E{w,lambda,2}} and \eqref{E{w,lambda,3}}, we derive
\begin{align*}
	\|E^{w,\lambda}(t)\|_{L^2(\mr^2)}\leq C\lambda^{-\min\{s+1-\delta, ~2s-1+\delta\}}.
\end{align*}
 The case $d=3$ can be handled similarly and we omit the details here. The proof of the lemma is finished.
\end{proof}
\subsection{Exact Solutions}\label{sec:3.3}
    When $d=2$, we choose $(u_{w,\lambda},b_{w,\lambda})$ to be the unique solution to the following resistive MHD equations
\begin{align}\label{mhd-solution-low+high}
	\begin{cases}
		\partial_t u_{w,\lambda}+u_{w,\lambda} \cdot \nabla u_{w,\lambda}+\nabla p_{w,\lambda}=b_{w,\lambda}\cdot \nabla b_{w,\lambda},\\[1mm]
		\partial_t b_{w,\lambda}-\Delta b_{w,\lambda}+u_{w,\lambda}\cdot \nabla b_{w,\lambda}=b_{w,\lambda}\cdot\nabla u_{w,\lambda},  \\[1mm]
		\diver\, u_{w,\lambda}=\diver\, b_{w,\lambda}=0, \\[1mm]
		u_{w,\lambda}(x, 0)=u^{w,\lambda}(x,0)=\nabla^\perp\left( \phi^{l,w,\lambda}(x)+\phi^{h,w,\lambda}(x,0)\right),\\[1mm]
		b_{w,\lambda}(x, 0)=b^{w,\lambda}(x,0)=\nabla^\perp\phi^{l,w,\lambda}(x),
	\end{cases}
\end{align}
It follows from \eqref{mhd-low+high} that
\begin{align}\label{mhd-low+high-appro}
	\begin{cases}
		\partial_t u^{w,\lambda}+u^{w,\lambda} \cdot \nabla u^{w,\lambda}+\nabla p^{w,\lambda}=b^{w,\lambda}\cdot \nabla b^{w,\lambda}+E^{w,\lambda},\\[1mm]
		\partial_t b^{w,\lambda}-\Delta b^{w,\lambda}+u^{w,\lambda}\cdot \nabla b^{w,\lambda}=b^{w,\lambda}\cdot\nabla u^{w,\lambda}+\text{div}\,\tilde{F}^{w,\lambda},  \\[1mm]
		\diver\, u^{w,\lambda}=\diver\, b^{w,\lambda}=0, \\[1mm]
		u^{w,\lambda}(x, 0)=\nabla^\perp\left( \phi^{l,w,\lambda}(x)+\phi^{h,w,\lambda}(x,0)\right),\,\, b^{w,\lambda}(x, 0)=\nabla^\perp\phi^{l,w,\lambda}(x),
	\end{cases}
\end{align}

Substracting \eqref{mhd-solution-low+high} from \eqref{mhd-low+high-appro} and using Lemma \ref{two-systems-E-F}, we derive
\begin{align*}
	&\left(\left\| \left(u^{w,\lambda}-u_{w,\lambda}\right)(t)\right\|_{L^2}+\left\| \left(b^{w,\lambda}-b_{w,\lambda}\right)(t)\right\|_{L^2}\right)\\
		&\leq t\left(\max_{0\leq \tau \leq t}\left(\|E^{w,\lambda}(\tau)\|_{L^2}^2+\|\tilde{F}^{w,\lambda}(\tau)\|_{L^2}^2\right)\right)^{\frac{1}{2}}	\exp\left(2t\max_{0\leq \tau \leq t}\left(\|\nabla u^{w,\lambda}(\tau)\|_{L^\infty} + \|\nabla b^{w,\lambda}(\tau)\|_{L^\infty}+1\right)\right).
\end{align*}
Combining \eqref{nabla-u-b-t-bound} with Lemma \ref{E,F-error-estimate}, we obtain
\begin{align}\label{u^w-w_w-b^w-b_w-lambda-decay}
		\left(\left\| \left(u^{w,\lambda}-u_{w,\lambda}\right)(t)\right\|_{L^2}+\left\| \left(b^{w,\lambda}-b_{w,\lambda}\right)(t)\right\|_{L^2}\right)\leq C\lambda^{-\sigma_{s,\delta}}.
\end{align}

For the case $d=3$, the  arguments are similar and we can obtain  an estimate analogous to \eqref{u^w-w_w-b^w-b_w-lambda-decay}.
\subsection{Proof of Theorem \ref{thm}: The Case
 ${\bf B_0} = \mathbf{0}$}\label{sec:3.4}

Now we give the proof of Theorem \ref{thm} in the case ${\bf B_0} = \mathbf{0}$. We will focus on the two-dimensional case $(d=2)$ and the three-dimensional case $(d=3)$ can be dealt with in a similar way.

\begin{proof}[Proof of Theorem \ref{thm}]
	 Using \eqref{phi-high} and \eqref{phi-low-Phi-cutoff} leads to
	\begin{align*}
		\|u_{w,\lambda}(0)\|_{H^s} + \|b_{w,\lambda}(0)\|_{H^s}
		&= \|u^{w,\lambda}(0)\|_{H^s} + \|b^{w,\lambda}(0)\|_{H^s} \\
		&\leq 2\|\nabla^{\perp}\phi^{l,w,\lambda}\|_{H^s} + \|\nabla^\perp \phi^{h,w,\lambda}(0)\|_{H^s} \\
		&\leq C\lambda^{-1+\delta} + C \leq C_1,
	\end{align*}
	where $C_1>0$ is a positive constant independent of $\lambda$ and $w=\pm 1$.
	
	To get \eqref{nonuniform-initial-bound} in Theorem \ref{thm}, we make a minor modification of the functions $\phi^{l,w,\lambda}$ and $\phi^{h,w,\lambda}$, which are
	\begin{align*}
		\phi^{h,w,\lambda}(x,t)=\begin{cases}
			\varepsilon\lambda^{-s-\delta-1}\phi(\dfrac{x_1}{\lambda^\delta})\phi(\dfrac{x_2}{\lambda^\delta})\sin (\lambda x_2-wt),&~d=2,\\[3mm]
			\varepsilon\lambda^{-s-\delta-1}\phi(\dfrac{x_1}{\lambda^\delta})\phi(\dfrac{x_2}{\lambda^\delta})\sin (\lambda x_2-wt)\phi(x_3),&~d=3,
		\end{cases}
	\end{align*}
and
\begin{align*}
	\phi^{l,w,\lambda}(x)=\begin{cases}
		-\varepsilon w\lambda^{-1+\delta}\Phi_1(\dfrac{x_1}{\lambda^\delta})\Phi_2(\dfrac{x_2}{\lambda^\delta}),&~d=2,\\[3mm]
		-\varepsilon w\lambda^{-1+\delta}\Phi_1(\dfrac{x_1}{\lambda^\delta})\Phi_2(\dfrac{x_2}{\lambda^\delta})\phi(x_3),&~d=3,
	\end{cases}
\end{align*}
where  $\varepsilon \leq \min\{1,\frac{\gamma}{C_1}\}$. Moreover, one has
\begin{align*}
	\|u_{w,\lambda}(0)\|_{H^s} + \|b_{w,\lambda}(0)\|_{H^s} \leq C_1\varepsilon\leq \gamma.
\end{align*}
 \eqref{nonuniform-initial-bound} is then proved.
	
	To prove \eqref{nonuniform-initial-diff}, according to \eqref{u-b-solution-initial},  the difference between the initial values can be expressed as
	\begin{align*}
		u_{1,\lambda}(0)-u_{-1,\lambda}(0)
		&=\nabla^\perp \left(\phi^{l,1,\lambda}(x)-\phi^{l,-1,\lambda}(x)\right)+\nabla^\perp \left(\phi^{h,1,\lambda}(x,0)-\phi^{h,-1,\lambda}(x,0)\right),\\
			b_{1,\lambda}(0)-b_{-1,\lambda}(0)&=\nabla^\perp \left(\phi^{l,1,\lambda}(x)-\phi^{l,-1,\lambda}(x)\right).
	\end{align*}
Thus, it yields
\begin{align*}
	&\left\|u_{1,\lambda}(0)-u_{-1,\lambda}(0)\right\|_{H^s}+	\left\|b_{1,\lambda}(0)-b_{-1,\lambda}(0)\right\|_{H^s}\\
	&\leq C\varepsilon\lambda^{-1}\left(\left\|\Phi'_1(\dfrac{\cdot}{\lambda^\delta})\right\|_{H^s(\mr)}\left\|\Phi_2(\dfrac{\cdot}{\lambda^\delta})\right\|_{H^s(\mr)}+\left\|\Phi_1(\dfrac{\cdot}{\lambda^\delta})\right\|_{H^s(\mr)}\left\|\Phi'_2(\dfrac{\cdot}{\lambda^\delta})\right\|_{H^s(\mr)}\right)\\
	&\leq C\varepsilon\lambda^{-1+\delta} ~\to~0, ~\text{for}~ \lambda\to \infty,
\end{align*}
which implies  \eqref{nonuniform-initial-diff}.

Now prove \eqref{nonuniform-low-bound-sint}. Let $\left(u_{1,\lambda}(t),b_{1,\lambda}(t)\right)$ and $\left(u_{-1,\lambda}(t),b_{-1,\lambda}(t)\right)$ be two sequences of solutions to the equations \eqref{mhd} corresponding to the following initial values \begin{align}\label{u-b-solution-initial}
\begin{split}
		\left(u_{1,\lambda}(0),b_{1,\lambda}(0)\right)=	\left(u^{1,\lambda}(0),b^{1,\lambda}(0)\right)&=\left(\nabla^\perp\left( \phi^{l,1,\lambda}(x)+\phi^{h,1,\lambda}(x,0)\right),\nabla^\perp\phi^{l,1,\lambda}(x)\right),\\
		\left(u_{-1,\lambda}(0),b_{-1,\lambda}(0)\right)=	\left(u^{-1,\lambda}(0),b^{-1,\lambda}(0)\right)&=\left(\nabla^\perp\left( \phi^{l,-1,\lambda}(x)+\phi^{h,-1,\lambda}(x,0)\right),\nabla^\perp\phi^{l,-1,\lambda}(x)\right),
\end{split}
\end{align}
 respectively. By Lemma \ref{local-well-pri},  we obtain
\begin{align*}
	\|u_{w,\lambda}(t),b_{w,\lambda}(t)\|_{H^m}&\leq C \|u_{w,\lambda}(0),b_{w,\lambda}(0)\|_{H^m}\\
	&=C \|u^{w,\lambda}(0),b^{w,\lambda}(0)\|_{H^m}\leq C\lambda^{-s+m},
\end{align*}
where  $w=\pm 1$ and $m>\frac{d}{2}+1$. Then it follows from \eqref{u,b-low-t-local-bound} and \eqref{u-high-bound} that
\begin{align}\label{5.1}
\begin{split}
				\|u_{1,\lambda}(t)-u^{1,\lambda}(t)\|_{H^m}+	\|b_{1,\lambda}(t)-b^{1,\lambda}(t)\|_{H^m}\leq  C\lambda^{-s+m},\\
			\|u_{-1,\lambda}(t)-u^{-1,\lambda}(t)\|_{H^m}+	\|b_{-1,\lambda}(t)-b^{-1,\lambda}(t)\|_{H^m}\leq  C\lambda^{-s+m}.
\end{split}
\end{align}

In \eqref{u^w-w_w-b^w-b_w-lambda-decay}, to guarantee  $\sigma_{s,\delta}>0$, the parameters must satisfy
\begin{equation}
	s > 1 - \delta \quad \text{and} \quad 0<\delta < 1,
\end{equation}
where $\delta$ represents the regularity loss exponent and $s$ is the Sobolev regularity index.

Using \eqref{u^w-w_w-b^w-b_w-lambda-decay} and \eqref{5.1}, with help of the interpolation, we obtain, for $w=\pm 1$,
\begin{align}\nonumber
	&\|u_{w,\lambda}(t)-u^{w,\lambda}(t)\|_{H^s}+\|b_{w,\lambda}(t)-b^{w,\lambda}(t)\|_{H^s}\\\nonumber
	&\leq C\left\|u_{w,\lambda}(t)-u^{w,\lambda}(t),b_{w,\lambda}(t)-b^{w,\lambda}(t)\right\|^{\frac{s}{s+\frac{d}{2}+1}}_{H^{s+\frac{d}{2}+1}}\left\|u_{w,\lambda}(t)-u^{w,\lambda}(t),b_{w,\lambda}(t)-b^{w,\lambda}(t)\right\|_{L^2}^{\frac{\frac{d}{2}+1}{s+\frac{d}{2}+1}}\\\nonumber
	&\leq C\lambda^{\left(\frac{d}{2}+1\right)\frac{s}{s+\frac{d}{2}+1}}\lambda^{-\sigma_{s,\delta}\frac{\frac{d}{2}+1}{s+\frac{d}{2}+1}}\\\label{u,b-appro-exact}
	&=C \lambda^{\left(s-\sigma_{s,\delta}\right)\frac{\frac{d}{2}+1}{s+\frac{d}{2}+1}},
\end{align}
where $\lambda$ is sufficiently large.

Under the following conditions
\begin{equation}\label{eq:parameter_condition}
	\frac{\frac{d}{2} + 1}{s + \frac{d}{2} + 1} >0 \quad \text{and} \quad (s - \sigma_{s,\delta})=\min\{-1+\delta,-s+1-\delta\} < 0,
\end{equation}
we have
\begin{equation}\label{eq:convergence_result}
	\lim_{\lambda \to \infty} \left( \|u_{w,\lambda}(t) - u^{w,\lambda}(t)\|_{H^s} + \|b_{w,\lambda}(t) - b^{w,\lambda}(t)\|_{H^s} \right) = 0.
\end{equation}

Note that
\begin{align}\label{u-high-diff}
	u^{h,1,\lambda}(t)-u^{h,-1,\lambda}(t)=\begin{pmatrix}
\begin{pmatrix}
				\lambda^{-s-\delta}\phi(\dfrac{x_1}{\lambda^\delta})\phi(\dfrac{x_2}{\lambda^\delta})\left(\cos (\lambda x_2-t)-\cos(\lambda x_2+t)\right)\\
			+\lambda^{-s-1-2\delta}\phi(\dfrac{x_1}{\lambda^\delta})\phi'(\dfrac{x_2}{\lambda^\delta})\left(\sin (\lambda x_2-t)-\sin (\lambda x_2+t)\right)
\end{pmatrix}\\[8mm]
			\left(- \lambda^{-s-1-2\delta}\phi'(\dfrac{x_1}{\lambda^\delta})\phi(\dfrac{x_2}{\lambda^\delta})\left(\sin (\lambda x_2-t)-\sin (\lambda x_2+t)\right)\right)
		\end{pmatrix}
\end{align}
and the first component $\left(u^{h,1,\lambda}(t)-u^{h,-1,\lambda}(t)\right)_1$ contains a relatively larger exponent $\lambda^{-s-\delta}$. It follows that
\begin{align*}
	&\left\|\lambda^{-s-\delta}\phi(\dfrac{x_1}{\lambda^\delta})\phi(\dfrac{x_2}{\lambda^\delta})\left(\cos (\lambda x_2-t)-\cos(\lambda x_2+t)\right)\right\|_{H^s}\\
	&=\left\|\lambda^{-s-\delta}\sin(\lambda x_2)\sin t\phi(\dfrac{x_1}{\lambda^\delta})\phi(\dfrac{x_2}{\lambda^\delta})\right\|_{H^s}\\
	&\geq |\sin t|\lambda^{-s-\delta}\left\|\phi(\dfrac{x_1}{\lambda^\delta})\right\|_{H^s(\mr)}\left\|\phi(\dfrac{x_2}{\lambda^\delta})\sin (\lambda x_2)\right\|_{H^s(\mr)}\\
	&\geq \tilde{c}|\sin t|.
\end{align*}
Applying \eqref{u-high-diff} yields
\begin{align}\nonumber
	&\left\|u^{h,1,\lambda}(t)-u^{h,-1,\lambda}(t)\right\|_{H^s}\\\nonumber
	&\geq \tilde{c}|\sin t|-\lambda^{-s-1-2\delta}\left\|\phi(\dfrac{x_1}{\lambda^\delta})\phi'(\dfrac{x_2}{\lambda^\delta})\left(\sin (\lambda x_2-t)-\sin (\lambda x_2+t)\right)\right\|_{H^s}\\\nonumber
	&~~~~~~-\lambda^{-s-1-2\delta}\left\|\phi'(\dfrac{x_1}{\lambda^\delta})\phi(\dfrac{x_2}{\lambda^\delta})\left(\sin (\lambda x_2-t)-\sin (\lambda x_2+t)\right)\right\|_{H^s}\\\label{u-high-diff-Hs}
	&\geq \tilde{c}|\sin t|-C\lambda^{-1-\delta}.
\end{align}

For any $0<t\leq T\leq 1$, using \eqref{u,b-appro-exact} leads to
\begin{align}\nonumber
	&\left\|u_{1,\lambda}(t)-u_{-1,\lambda}(t)\right\|_{H^s}+	\left\|b_{1,\lambda}(t)-b_{-1,\lambda}(t)\right\|_{H^s}\\\nonumber
	&\geq \left\| u^{1,\lambda}(t)-u^{-1,\lambda}(t),b^{1,\lambda}(t)-b^{-1,\lambda}(t)\right\|_{H^s}\\\nonumber
	&~~~~~~-\left\| u^{1,\lambda}(t)-u_{1,\lambda}(t),b^{1,\lambda}(t)-b_{1,\lambda}(t)\right\|_{H^s}\\\nonumber
		&~~~~~~-\left\| u^{-1,\lambda}(t)-u_{-1,\lambda}(t),b^{-1,\lambda}(t)-b_{-1,\lambda}(t)\right\|_{H^s}\\\label{5.2}
		&\geq \left\| u^{1,\lambda}(t)-u^{-1,\lambda}(t),b^{1,\lambda}(t)-b^{-1,\lambda}(t)\right\|_{H^s}-C\varepsilon\lambda^{\left(s-\sigma_{s,\delta}\right)\frac{\frac{d}{2}+1}{s+\frac{d}{2}+1}}.
\end{align}
Due to \eqref{u,b-low-t-local-bound} and \eqref{u-high-diff-Hs}, it holds that
\begin{align*}
	&\left\| u^{1,\lambda}(t)-u^{-1,\lambda}(t),b^{1,\lambda}(t)-b^{-1,\lambda}(t)\right\|_{H^s}\\
	&\geq \left\|u^{h,1,\lambda}(t)-u^{h,-1,\lambda}(t)\right\|_{H^s}- \left\|u^{l,1,\lambda}(t)\right\|_{H^s}-\left\|u^{l,-1,\lambda}(t)\right\|_{H^s}\\
	&~~~~~~~~~~-\left\|b^{1,\lambda}(t)\right\|_{H^s}-\left\|b^{-1,\lambda}(t)\right\|_{H^s}\\
	&\geq \tilde{c}\varepsilon|\sin t|-C\varepsilon\lambda^{-1-\delta}-C\varepsilon\lambda^{-1+\delta}\geq \tilde{c}\varepsilon|\sin t|-C\lambda^{-1+\delta}.
\end{align*}
Substituting  into \eqref{5.2} gives
\begin{align*}
	&\left\|u_{1,\lambda}(t)-u_{-1,\lambda}(t)\right\|_{H^s}+	\left\|b_{1,\lambda}(t)-b_{-1,\lambda}(t)\right\|_{H^s}\\
	&\geq \tilde{c}\varepsilon|\sin t|-C\varepsilon\lambda^{-1+\delta}-C\varepsilon\lambda^{\left(s-\sigma_{s,\delta}\right)\frac{\frac{d}{2}+1}{s+\frac{d}{2}+1}} ~\to~\tilde{c}\varepsilon|\sin t|, ~\text{for}~ \lambda\to \infty.
\end{align*}
Now, choosing $C_2$ such that $\frac{\gamma}{C_2}<\min\{1,\frac{\gamma}{C_1}\}$ and $\frac{\gamma}{C_2}<\varepsilon\leq \min\{1,\frac{\gamma}{C_1}\}$, we obtain $\tilde{c}\varepsilon|\sin t| > \frac{\tilde{c}}{C_2}\gamma |\sin t|\define c\gamma |\sin t|$, which implies \eqref{nonuniform-low-bound-sint}.

Up to now Theorem \ref{thm} is proved in the case $d=2$. The case $d=3$ can be proved in a similar way and hence we finish the proof of Theorem \ref{thm}.
\end{proof}

\section{The Case ${\bf B_0} \neq \mathbf{0}$}\label{sec:4}
~~~~In this section, we give a sketch of proof Theorem \ref{thm} in the case ${\bf B_0}=(B_1, B_2)\neq {\bf 0}$. We will focus on the two-dimensional case and three-dimensional case can be treated in a similar way.

\subsection{The Approximate Solution}\label{sec:4.1}
In two-dimensional case, we assume that ${\bf B_0}=(B_1, B_2)\neq {\bf 0}$. Based on \eqref{phi-high} and \eqref{phi-low-Phi-cutoff}, using a coordinate transformation
\begin{align*}
	x_2\mapsto B_1x_2-B_2x_1,
\end{align*}
we can choose the stream function as follows
\begin{align}\label{4.1}
		\bar{\phi}^{h,w,\lambda}(x,t)=
		\lambda^{-s-\delta-1}\phi\left(\dfrac{x_1}{\lambda^\delta}\right)\phi\left(\dfrac{B_1x_2-B_2x_1}{\lambda^\delta}\right)\sin \left(\lambda \left(B_1x_2-B_2x_1\right)-wt\right),
\end{align}
and
\begin{align}\label{4.2}
		\bar{\phi}^{l,w,\lambda}(x)=
		-w\lambda^{-1+\delta}\Phi_1\left(\dfrac{x_1}{\lambda^\delta}\right)\Phi_2\left(\dfrac{B_1x_2-B_2x_1}{\lambda^\delta}\right).
\end{align}

The  high-frequency term $\bar{u}^{w,\lambda}$ is then defined as
\begin{align}\nonumber
	&\bar{u}^{h,w,\lambda}(x,t)=\nabla^\perp \bar{\phi}^{h,w,\lambda}(x,t)\\\label{4.4}
	&=\begin{pmatrix}
\begin{pmatrix}
			B_1\lambda^{-s-\delta}\phi\left(\dfrac{x_1}{\lambda^\delta}\right)\phi\left(\dfrac{B_1x_2-B_2x_1}{\lambda^\delta}\right)\cos \left(\lambda \left(B_1x_2-B_2x_1\right)-wt\right)\\[3mm]
		+B_1\lambda^{-s-1-2\delta}\phi\left(\dfrac{x_1}{\lambda^\delta}\right)\phi'\left(\dfrac{B_1x_2-B_2x_1}{\lambda^\delta}\right)\sin \left(\lambda \left(B_1x_2-B_2x_1\right)-wt\right)
\end{pmatrix}\\[10mm]
\begin{pmatrix}
			- \lambda^{-s-1-2\delta}\phi'\left(\dfrac{x_1}{\lambda^\delta}\right)\phi\left(\dfrac{B_1x_2-B_2x_1}{\lambda^\delta}\right)\sin \left(\lambda \left(B_1x_2-B_2x_1\right)-wt\right)\\[3mm]
	+B_2\lambda^{-s-\delta}\phi\left(\dfrac{x_1}{\lambda^\delta}\right)\phi\left(\dfrac{B_1x_2-B_2x_1}{\lambda^\delta}\right)\cos \left(\lambda \left(B_1x_2-B_2x_1\right)-wt\right)\\[3mm]
	+B_2\lambda^{-s-1-2\delta}\phi\left(\dfrac{x_1}{\lambda^\delta}\right)\phi'\left(\dfrac{B_1x_2-B_2x_1}{\lambda^\delta}\right)\sin \left(\lambda \left(B_1x_2-B_2x_1\right)-wt\right)
\end{pmatrix}
	\end{pmatrix}
\end{align}
Similar to \eqref{3.1}, for any $r\geq 0$, one has
\begin{align}\nonumber
	&\left\| \lambda^{-s-\delta}\phi\left(\dfrac{x_1}{\lambda^\delta}\right)\phi\left(\dfrac{B_1x_2-B_2x_1}{\lambda^\delta}\right)\cos \left(\lambda \left(B_1x_2-B_2x_1\right)-wt\right)\right\|_{H^r(\mr^2)}\\\label{4.3}
	&\leq C\lambda^{-s-\delta}\left\| \phi(\frac{\cdot}{\lambda^{\delta}})\right\|_{H^r(\mr)}\left\| \phi(\frac{\cdot}{\lambda^{\delta}})\cos (\lambda\cdot-wt)\right\|_{H^r(\mr)}\leq C\lambda^{r-s}.
\end{align}

  The low-frequency component $(\bar{u}^{l,w,\lambda}, \bar{b}^{w,\lambda})$ is chosen to be  the solution to system \eqref{mhd-B} with initial data
\begin{align*}
	\bar{u}^{l,w,\lambda}(x,0) = \bar{b}^{w,\lambda}(x,0) = \nabla^\perp \bar{\phi}^{l,w,\lambda}(x).
\end{align*}

Following similar estimates established in \eqref{u,b-low-t-local-bound} and \eqref{u-high-bound}, we can obtain
\begin{align}\label{4.5}
		\|\bar{u}^{l,w,\lambda}(t)\|_{H^m} + \|\bar{b}^{w,\lambda}(t)\|_{H^{m}}
		\leq C\lambda^{-1+\delta},
\end{align}
and
\begin{align}\label{4.6}
	\|\bar{u}^{h,w,\lambda}(t)\|_{H^r}\leq C\lambda^{r-s},\quad \|\bar{u}^{h,w,\lambda}(t)\|_{L^\infty}\leq C\lambda^{-s-\delta},\quad \|\nabla \bar{u}^{h,w,\lambda}(t)\|_{L^\infty}\leq C\lambda^{-s+1-\delta}
\end{align}
for any $0 \leq t \leq T\leq 1$, where $[0,T]$ denotes the existence interval of the solution guaranteed by Lemma~\ref{local-well-pri}.
\subsection{Error Estimates for the Approximate Solutions}\label{sec:4.2}
Following the approach in subsection \ref{sec:3.2} and based on system \eqref{mhd-B}, we construct a perturbed system  as follows
\begin{align}\label{4.9}
		\begin{cases}
		\partial_t \bar{u}^{w,\lambda}+\bar{u}^{w,\lambda} \cdot \nabla \bar{u}^{w,\lambda}+\nabla p^{w,\lambda}=\bar{b}^{w,\lambda}\cdot \nabla \bar{b}^{w,\lambda}+{\bf B_0}\cdot\nabla \bar{b}^{w,\lambda}+E_1^{w,\lambda},\\[1mm]
		\partial_t \bar{b}^{w,\lambda}-\Delta \bar{b}^{w,\lambda}+\bar{u}^{w,\lambda}\cdot \nabla \bar{b}^{w,\lambda}=\bar{b}^{w,\lambda}\cdot\nabla \bar{u}^{w,\lambda}+{\bf B_0}\cdot\nabla \bar{u}^{w,\lambda}+F_1^{w,\lambda},  \\[1mm]
		\diver\, \bar{u}^{w,\lambda}=\diver\, \bar{b}^{w,\lambda}=0,
	\end{cases}
\end{align}
where
\begin{align*}
		E_1^{w,\lambda}&=\partial_t \bar{u}^{h,w,\lambda}+\bar{u}^{l,w,\lambda}\cdot\nabla \bar{u}^{h,w,\lambda}+\bar{u}^{h,w,\lambda}\cdot\nabla \bar{u}^{h,w,\lambda}+\bar{u}^{h,w,\lambda}\cdot\nabla \bar{u}^{l,w,\lambda},\\\nonumber
	F_1^{w,\lambda}&=\bar{u}^{h,w,\lambda}\cdot\nabla \bar{b}^{w,\lambda}-\bar{b}^{w,\lambda}\cdot\nabla \bar{u}^{h,w,\lambda}-{\bf B_0}\cdot \nabla \bar{u}^{h,w,\lambda}\\
	&\define\diver \tilde{F}_1^{w,\lambda}-{\bf B_0}\cdot \nabla \bar{u}^{h,w,\lambda}.
\end{align*}
The perturbed system \eqref{4.9} is  analogous to \eqref{mhd-low+high}. The directional derivative of the high-frequency component along the background magnetic field ${\bf B_0} = (B_1, B_2)$ is
\begin{align*}
	{\bf B_0} \cdot \nabla \bar{u}^{h,w,\lambda} = B_1\partial_{x_1}\bar{u}^{h,w,\lambda} + B_2\partial_{x_2}\bar{u}^{h,w,\lambda}.
\end{align*}
Remarkably, we observe that the oscillatory part in $\bar{u}^{h,w,\lambda}$ satisfies
\begin{align*}
	{\bf B_0} \cdot \nabla \left[\phi\left(\frac{B_1x_2 - B_2x_1}{\lambda^\delta}\right) \cos\left(\lambda(B_1x_2 - B_2x_1) - wt\right)\right] = 0,
\end{align*}
which leads to the following simplified expression
\begin{align*}
	&	{\bf B_0}\cdot \nabla	\left(\lambda^{-s-\delta}\phi\left(\dfrac{x_1}{\lambda^\delta}\right)
	\phi\left(\dfrac{B_1x_2-B_2x_1}{\lambda^\delta}\right)\cos \left(\lambda \left(B_1x_2-B_2x_1\right)-wt\right)\right)\\
	&=B_1\lambda^{-s-2\delta}\phi'\left(\frac{x_1}{\lambda^{\delta}}\right)	\phi\left(\dfrac{B_1x_2-B_2x_1}{\lambda^\delta}\right)\cos \left(\lambda \left(B_1x_2-B_2x_1\right)-wt\right)
\end{align*}
Therefore, following analogous arguments as  \eqref{4.3}, we can obtain
\begin{align}\label{4.7}
	\left\| {\bf B_0}\cdot\nabla \bar{u}^{h,w,\lambda}\right\|_{L^2}\leq C\lambda^{-s-\delta}.
\end{align}
Combining the estimates \eqref{4.5}, \eqref{4.6} with Lemma \ref{E,F-error-estimate}, we deduce that
\begin{align}\label{4.8}
	\|E_1^{w,\lambda}(t)\|_{L^2}\leq C\lambda^{-\sigma_{s,\delta}},\quad	\|\tilde{F}_1^{w,\lambda}(t)\|_{L^2}\leq C\lambda^{-s-1},
\end{align}
where \begin{align*}
	\sigma_{s,\delta}=\min\{s+1-\delta,2s-1+\delta\}.
\end{align*}
\subsection{Exact Solutions}\label{sec:4.3}
Let $(\bar{u}_{w,\lambda},\bar{b}_{w,\lambda})$ be the unique solution to system \eqref{mhd-B} with the initial data $\left(\bar{u}^{w,\lambda}(x,0),\bar{b}^{w,\lambda}(x,0)\right)$, satisfying
\begin{align*}
	\begin{cases}
		\partial_t \bar{u}_{w,\lambda}+\bar{u}_{w,\lambda} \cdot \nabla \bar{u}_{w,\lambda}+\nabla p_{w,\lambda}=\bar{b}_{w,\lambda}\cdot \nabla \bar{b}_{w,\lambda}+{\bf B_0}\cdot \nabla \bar{b}_{w,\lambda},\\[1mm]
		\partial_t \bar{b}_{w,\lambda}-\Delta \bar{b}_{w,\lambda}+\bar{u}_{w,\lambda}\cdot \nabla \bar{b}_{w,\lambda}=\bar{b}_{w,\lambda}\cdot\nabla \bar{u}_{w,\lambda}+{\bf B_0}\cdot \nabla \bar{u}_{w,\lambda},  \\[1mm]
		\diver\, \bar{u}_{w,\lambda}=\diver\, \bar{b}_{w,\lambda}=0,
	\end{cases}
\end{align*}
It follows from \eqref{4.9} that
\begin{align*}
	\begin{cases}
	\partial_t \bar{u}^{w,\lambda}+\bar{u}^{w,\lambda} \cdot \nabla \bar{u}^{w,\lambda}+\nabla p^{w,\lambda}=\bar{b}^{w,\lambda}\cdot \nabla \bar{b}^{w,\lambda}+{\bf B_0}\cdot\nabla \bar{b}^{w,\lambda}+E_1^{w,\lambda},\\[1mm]
	\partial_t \bar{b}^{w,\lambda}-\Delta \bar{b}^{w,\lambda}+\bar{u}^{w,\lambda}\cdot \nabla \bar{b}^{w,\lambda}=\bar{b}^{w,\lambda}\cdot\nabla \bar{u}^{w,\lambda}+{\bf B_0}\cdot\nabla \bar{u}^{w,\lambda}+\diver \tilde{F}_1^{w,\lambda}-{\bf B_0}\cdot \nabla \bar{u}^{h,w,\lambda},  \\[1mm]
	\diver\, \bar{u}^{w,\lambda}=\diver\, \bar{b}^{w,\lambda}=0.
\end{cases}
\end{align*}
The two systems share same initial conditions, which are
\begin{align*}
	\bar{u}_{w,\lambda}(x, 0)&=\bar{u}^{w,\lambda}(x,0)=\nabla^\perp\left( \bar{\phi}^{l,w,\lambda}(x)+\bar{\phi}^{h,w,\lambda}(x,0)\right),\\[1mm]
	\bar{b}_{w,\lambda}(x, 0)&=\bar{b}^{w,\lambda}(x,0)=\nabla^\perp\bar{\phi}^{l,w,\lambda}(x)
\end{align*}

By analyzing the difference between these systems and applying Lemma \ref{two-systems-E-F}, we obtain
\begin{align*}
	&\left(\left\| \left(\bar{u}^{w,\lambda}-\bar{u}_{w,\lambda}\right)(t)\right\|_{L^2}+\left\| \left(\bar{b}^{w,\lambda}-\bar{b}_{w,\lambda}\right)(t)\right\|_{L^2}\right)\\
	&\leq t\left(\max_{0\leq \tau \leq t}\left(\|E_1^{w,\lambda}(\tau)\|_{L^2}^2+\|\tilde{F}_1^{w,\lambda}(\tau)\|_{L^2}^2+\|{\bf B_0}\cdot\nabla \bar{u}^{h,w,\lambda}(\tau)\|_{L^2}^2\right)\right)^{\frac{1}{2}}	\\
	&~~~~~~~~~\times\exp\left(2\max_{0\leq \tau \leq t}\left(\|\nabla \bar{u}^{w,\lambda}(\tau)\|_{L^\infty} + \|\nabla \bar{b}^{w,\lambda}(\tau)\|_{L^\infty}+1\right)\right).
\end{align*}
Combining the estimates \eqref{4.5}, \eqref{4.6}, \eqref{4.8}, and \eqref{4.9}, we establish the fundamental error bound:
\begin{align}
	\big\|(\bar{u}^{w,\lambda} - \bar{u}_{w,\lambda})(t)\big\|_{L^2} + \big\|(\bar{b}^{w,\lambda} - \bar{b}_{w,\lambda})(t)\big\|_{L^2} \leq C\lambda^{-\bar{\sigma}_{s,\delta}},
\end{align}
where the convergence rate is given by
\begin{align*}
	\bar{\sigma}_{s,\delta} \define \min\{s+\delta, \sigma_{s,\delta}\} = \min\{s+\delta, s+1-\delta, 2s-1+\delta\}.
\end{align*}

\subsection{Proof of Theorem \ref{thm}: The Case ${\bf B_0} \neq \mathbf{0}$}\label{sec:4.4}
 Now we finish the proof of Theorem \ref{thm} in the case ${\bf B_0} \neq \mathbf{0}$. We focus on verifying the following two essential points:

{\bf 1. Regularity Condition}

Following the regularity pattern in \eqref{u,b-appro-exact}, we have
\begin{align}\label{0704}
	(s - \bar{\sigma}_{s,\delta}) = \max\{-\delta, s - \sigma_{s,\delta}\} < 0.
\end{align}
In fact, the parameter condition \eqref{eq:parameter_condition} guarantees that there exists  $\delta$ such that \eqref{0704} holds.

{\bf 2. High-Frequency Analysis}

Denote $\Theta_\lambda = \lambda (B_1x_2-B_2x_1)$. From \eqref{u-high-diff},  the oscillatory component satisfies
\begin{align*}
	&\left\|\lambda^{-s-\delta}\phi\left(\dfrac{x_1}{\lambda^\delta}\right)\phi\left(\dfrac{B_1x_2-B_2x_1}{\lambda^\delta}\right)\left[\cos (\Theta_\lambda - t) - \cos(\Theta_\lambda + t)\right]\right\|_{H^s(\mathbb{R}^2)}\\
	&= 2\left\|\lambda^{-s-\delta}\sin(\Theta_\lambda)\sin t\,\phi\left(\dfrac{x_1}{\lambda^\delta}\right)\phi\left(\dfrac{B_1x_2-B_2x_1}{\lambda^\delta}\right)\right\|_{H^s(\mathbb{R}^2)}\\
	&\geq 2B_1|\sin t|\lambda^{-s-\delta}\left\|\phi\left(\dfrac{x_1}{\lambda^\delta}\right)\right\|_{H^s(\mathbb{R})}\left\|\phi\left(\dfrac{x_2}{\lambda^\delta}\right)\sin (\lambda x_2)\right\|_{H^s(\mathbb{R})}\\
	&\geq \tilde{c}|\sin t|,
\end{align*}
which leads to the principal high-frequency component estimate
\begin{align*}
	|\bar{u}^{h,1,\lambda}(t) - \bar{u}^{h,-1,\lambda}(t)|_{H^s} \geq \tilde{c}|\sin t| - C\lambda^{-1-\delta},
\end{align*}
where the constants $\tilde{c}, C > 0$ are independent of $\lambda$.

 Following the similar approach as in subsection \ref{sec:3.4}, we can extend the proof of Theorem \ref{thm} in the case ${\bf B_0} =\mathbf{0}$ to the case  ${\bf B_0} \neq \mathbf{0}$. The details are omitted and the proof of Theorem \ref{thm} is finished.

			%\vspace{2mm}
			%{\bf Conflict of interest:}% The authors declare that they have no conflict of interest.
			%\vspace{2mm}
			
			%\section*{Ackonwledgments}

		%	\section*{Data availability statement}
		%	\,\,\,\,\,\,\,\,Data availability is not applicable to this article as no data sets were created or analysed in this study.

			\end{document}